\documentclass[11pt,twoside]{amsart}

\usepackage{csquotes}

\usepackage{amsthm,amsfonts,amssymb,amsmath,latexsym,verbatim,amscd,mathrsfs,dsfont}
\newtheorem{thm}{\bf Theorem}[section]
\newtheorem{prop}[thm]{\bf Proposition}
\newtheorem{lemma}[thm]{\bf Lemma}
\newtheorem*{Def}{\bf Definition}
\newtheorem{rmk}[thm]{\bf Remark}

\newtheorem{cor}[thm]{\bf Corollary}

\def\R{\mathbb R}
\def\N{\mathbb N}
\def\Z{\mathbb Z}
\def\d{\partial}
\def\ep{\varepsilon}
\renewcommand\div{{\rm div}\,}
\def\wt{\widetilde}
\def\cC{\mathcal C}
\def\cE{\mathcal E}
\def\sC{\mathscr C}

\def\cD{\mathcal D}
\def\sH{\mathscr H}
\def\sK{\mathscr K}
\def\cM{\mathcal M}

\def\cS{\mathcal S}
\def\cT{\mathcal T}

\def\hB{\dot{B}}

\def\hE{\dot{E}}
\def\hF{\dot{F}}
\def\hR{\dot{R}}
\def\hS{\dot{S}}
\def\hT{\dot{T}}
\def\hcT{\dot{\mathcal T}}
\def\hsC{\dot{\mathscr C}}

\def\hDelta{\dot{\Delta}}

\def\tf{\widetilde{f}}
\def\tX{\widetilde{X}}

\def\thDelta{\widetilde{\dot{\Delta}}}
\numberwithin{equation}{section}
\usepackage{cleveref}

\begin{document}
\title[Density patches in the inhomogeneous Navier-Stokes equations]{On the persistence of H\"older regular 
patches of density for the  inhomogeneous Navier-Stokes equations}
\author{Rapha\"el DANCHIN and Xin ZHANG}
\keywords{Inhomogeneous Navier-Stokes equations; $\mathcal{C}^{1, \varepsilon}$ density patch; Striated regularity}
\email{raphael.danchin@u-pec.fr and xin.zhang@univ-paris-est.fr}
\address{Universit\'{e} Paris-Est,  LAMA (UMR 8050), UPEMLV, UPEC, CNRS,  61 avenue du G\'en\'eral de Gaulle, 94010 Cr\'eteil Cedex 10. }

\begin{abstract}  In  our recent work dedicated to   the Boussinesq equations \cite{DanZhx2016}, 
we established the persistence of solutions with piecewise constant temperature 
along interfaces with H\"older regularity. 
We here address the same problem for   the inhomogeneous Navier-Stokes equations satisfied by a viscous incompressible
and inhomogeneous fluid. We establish that, indeed, in the slightly inhomogeneous case, patches of densities with $\mathcal{C}^{1, \varepsilon}$ regularity 
propagate for all time. 

As in  \cite{DanZhx2016},   our result follows from the conservation of H\"older regularity
along vector fields moving with the flow. The proof of that latter result  is based on commutator estimates involving  para-vector fields, and multiplier spaces. The overall analysis is more complicated than in \cite{DanZhx2016} 
 however, since the coupling between the mass and velocity equations in 
 the inhomogeneous Navier-Stokes equations is \emph{quasilinear} while it is linear  for the Boussinesq equations.
  \end{abstract}
\maketitle
%%%%%%%%%%%%%%%%%%%%%%%%%%%%%%%%%%%%
\section*{Introduction}
We are concerned with  the following \emph{inhomogeneous incompressible Navier-Stokes equations} in the whole space $\R^N$ with $N\geq2$:
\begin{equation}\tag{${INS}$}\label{eq:INSmu}
\left\{\begin{array}{l}
\d_t\rho+u\cdot\nabla\rho=0,\\
\rho(\d_tu+u\cdot\nabla u)-\mu\Delta u+\nabla P=0,\\
\div u=0,\\
(\rho,u)|_{t=0} = (\rho_0, u_0).
\end{array}\right.
\end{equation}
Above,  the unknowns 
%\footnote{Here we  denote $\R_+ := [0, \infty[.$}
$(\rho, u, P) \in \R_+ \times \R^{N} \times \R$ stand for the density, velocity vector field and pressure, respectively, and the so-called viscosity coefficient $\mu$ is a positive constant. 
\medbreak
There is an important literature dedicated to  the mathematical analysis of System \eqref{eq:INSmu}. 
The global existence of finite energy weak solutions with no vacuum (i.e. $\rho>0$) 
has been established in the seventies (see the monograph \cite{AKM1990}
and the references therein), then  extended by {\scshape  Simon} in \cite{Sim1990} in the  vacuum case.
Similar results have been obtained shortly after by {\scshape Lions} in the more general case where
the viscosity is density-dependent (see \cite{LionsP1996}). 

 Among the numerous open questions raised by {\scshape Lions} in  \cite{LionsP1996}, 
 the so-called \emph{density  patch problem} is a particularly challenging one. 
 The question is whether,  assuming that $\rho_0 =\mathds{1}_{\mathcal{D}_0}$
 for some domain  $\mathcal{D}_0$ of $\R^2$
 and that $\sqrt{\rho_0}\, u_0$  is in $L^2(\R^2),$ it is true that
 we have
 \begin{equation}\label{eq:rhott}
 \rho(t)=\mathds{1}_{\mathcal{D}_t}\quad\hbox{for all   }\ t\geq0
 \end{equation}
  for some domain $\mathcal{D}_t$
 \emph{with the same regularity as the initial one.}
 Although the renormalized solutions theory of {\scshape Di Perna} and {\scshape Lions} \cite{DPL} for transport equations
 ensures that we do have  \eqref{eq:rhott} with
  $\mathcal{D}_t$  being  the image of  $\cD_0$ by the volume preserving (generalized) flow of $u,$
    the weak solution framework does not give much information on 
    the regularity of the patch $\cD_t$ for positive times. 
 \medbreak
 The present paper aims at making one more step toward solving 
  {\scshape Lions'} question, by considering the case where 
  \begin{equation}\label{eq:rho0}
  \rho_0 = \eta_1 \mathds{1}_{\mathcal{D}_0} + \eta_2 \mathds{1}_{\mathcal{D}_0^c},\end{equation}
  for some simply connected bounded domain $\mathcal{D}_0$ of
  class $\cC^{1,\ep},$ and positive constants $\eta_1$ and $\eta_2$ \emph{close to one another}.   
    \medbreak
    
  That issue has been considered recently  in  \cite{LiaoZh2016a, LiaoZh2016b} by {\scshape Liao and Zhang} in  the 2-D case (see also \cite{LiaoL2016} for the 3-D case), first assuming 
  that $|\eta_1-\eta_2|$ is small then in the more challenging case where 
  $\eta_1$ and $\eta_2$ are \emph{any} positive real numbers.
  Under suitable striated-type regularity assumptions for the initial velocity, 
  the authors proved   the all-time persistence of high Sobolev
  regularity of patches of density.
  \smallbreak
  Before giving more insight into our main results, let  us briefly recall how {\scshape Liao and Zhang's}  proof goes. As in the pioneering work by {\scshape Chemin} \cite{Che1993} dedicated to the vortex patches problem for the 2-D incompressible Euler equations, the regularity of the interfaces is described by means of one (or several) tangent vector fields that evolve according to the flow of the velocity field. More precisely, let us assume that the  boundary  $\d\cD_0$ of  the initial patch 
   $\cD_0$    is  the level set $f_0^{-1}(\{0\})$ of some
  function $f_0:\R^2\to\R$ that does not degenerate in a neighborhood of $\d\cD_0.$ 
  Then the vector field   $X_0:= \nabla^\perp f_0$ is tangent to $\d\cD_0.$
  Now, if we denote by $\psi$ the flow associated to the velocity field $u,$ that 
  is the solution to the (integrated) ordinary differential equation
  \begin{equation}\label{eq:flow}
  \psi(t,x)=x+\int_0^t u\big(\tau,\psi(\tau,x)\big)\,d\tau,
  \end{equation}
  then  the boundary of ${\mathcal{D}_t}:=\psi(t,{\mathcal{D}_0})$ coincides with  $f_t^{-1}(\{0\})$ where
  $f_t:=f_0\circ \psi_t^{-1}$ and $\psi_t:=\psi(t,\cdot),$ and    
   we have
   \begin{equation}\label{eq:rhot}
  \rho(t,\cdot) = \eta_1 \mathds{1}_{\mathcal{D}_t} + \eta_2 \mathds{1}_{\mathcal{D}_t^c}.
  \end{equation}
  Note that  the tangent vector field
  $X_t:=\nabla^\perp f_t$  coincides with the evolution of the 
  initial vector field $X_0$ along the flow of $u,$ that 
  is\footnote{For any vector field  $Y = Y^{k}(x) \d_k$  and  function $f$ in $\cC^1(\R^N;\R),$
we denote by $\d_Y f$  the \emph{directional derivative} of $f$ along  $Y$, that is, 
with the  Einstein summation convention, $\d_Y f :=  Y^k \d_k f=Y\cdot \nabla f.$}~:
\begin{equation}\label{eq:Xt}X(t,\cdot) := (\d_{X_0} \psi)\circ \psi_t^{-1},\end{equation}
and thus satisfies, at least formally, the transport equation 
\begin{equation}\label{eq:X}
\left\{
\begin{array}{l}
\d_t X +u\cdot \nabla X = \d_{X} u, \\
 X|_{t=0}=X_0.
\end{array}
\right.
\end{equation}
\medbreak

Consequently, the problem of persistence of  regularity for the patch reduces to that of 
the vector field $X$ solution to \eqref{eq:X}. 
In their outstanding work, {\scshape Liao and Zhang} justified that  heuristics 
in the case of high Sobolev regularity, first if  $\eta_1$ and $\eta_2$ are close to one another  \cite{LiaoZh2016a},  and next assuming
only that $\eta_1$ and $\eta_2$ are positive  \cite{LiaoZh2016b}.  More precisely, the function $f_0$ is assumed to be   in $W^{k,p}(\R^2)$  for
some integer number  $k\geq3$  and real number $p$ in $]2,4[,$ 
and the initial velocity field $u_0,$ to  satisfy
the following \emph{striated regularity} property along the vector field $X_0:=\nabla^\perp f_0$:
$$
\d_{X_0}^\ell u_0 \in  H^{s-\frac{\ep\ell}k}\ \hbox{ for all }\ \ell\in\{0,\cdots,k\}\
\hbox{ and for some }\  0<\ep<s<1.
$$
Note however that the minimal regularity requirement in \cite{LiaoZh2016a,LiaoZh2016b} is 
that  $f$ is in   $W^{3,p}$ for some $p\in]2,4f[.$
In terms of H\"older inequality, this means (using Sobolev embedding), 
 that  the boundary of the patch  must be at least in $\cC^{2,\ep}$ for some  $\ep>0.$ 
  \smallbreak
In order to propagate lower order H\"older regularity, one may take advantage of the recent results by {\scshape Huang, Paicu and Zhang} in 
\cite{HPZ2013} (see also \cite{DZ}). Indeed, there, for small enough $\delta>0,$ the authors  construct global unique solutions 
with flow in   $\cC^{1,\delta}$ whenever  the initial density is close enough (for the $L^\infty$ norm)
to some positive constant, and  $u_0$  is in the Besov space 
 $\hB^{\frac{N}{p} -1}_{p,1}(\R^N)\cap \hB^{\frac{N}{p} + \delta-1}_{p,1}(\R^N)$ (see the definition below in 
  \eqref{def:besov}).  This clearly allows to propagate  $\cC^{1,\ep}$ interfaces, but only for $\ep\leq\delta,$ because  the maximal value of $\ep$ is limited by the global regularity assumption on $u_0$    
   although  {\scshape Liao and Zhang}'s results  mentioned above (as well as those
  of {\scshape Chemin} \cite{Che1993} in the context of Euler equations) suggest
 that only  tangential   regularity   is needed to  propagate the regularity of the patch.

%%%%%%%%%%%%%%%%%%%%%%%%%%%%%%%%%%%%%%%%%%

\section{Results}

Our goal here is to propagate  the  $\cC^{1,\ep}$ H\"older 
 regularity of the patch, within a \emph{critical} regularity framework. 
 By critical, we mean that we strive for a solution space 
 having the same scaling invariance by time and space dilations as
  \eqref{eq:INSmu} itself, namely: 
 \begin{equation}\label{eq:critical}
(\rho, u, P)(t,x) \rightarrow (\rho, \lambda u, \lambda^2 P)(\lambda^2 t , \lambda x)\quad\hbox{and}\quad
 (\rho_0, u_0)(x) \rightarrow (\rho_0, \lambda u_0)(\lambda x).
\end{equation}
Working with critical regularity is by now a classical approach 
for the  homogeneous Navier-Stokes equations (that is  $\rho$ is a positive constant in \eqref{eq:INSmu})
in the whole space $\R^N$  (see e.g. \cite{BCD2011,L-R2016} and the references
therein) and that it  is also  relevant 
in the inhomogeneous  situation  (see in particular the work by the first author in \cite{Dan2003} devoted
to the well-posedness issue in critical homogeneous Besov spaces, and its generalization 
to more general Besov spaces performed  by  {\scshape Abidi} in \cite{Ab2007} and {\scshape Abidi and Paicu} in \cite{AbP2007}). 
  % \cite{Dan2003} that  \eqref{eq:INSmu} has a global unique solution in 
   %dimension  $N \geq 3$ if supplemented with data such that, for a
    %small enough constant\footnote{If $(A, \|\cdot\|_{B})$ and $(B, \|\cdot\|_{B})$ are two Banach spaces and both of them belong to some Hausdorff topological vector space, then we denote $\big(A \cap B, \|\cdot\|_{A\cap B} \big)$ is a new Banach space with $\|\cdot\|_{A\cap B}:=\|\cdot\|_{A}+\|\cdot\|_{B}.$  } $c,$
%\begin{equation*}
%\|\rho_0-1\|_{\dot{B}^{\frac{N}{2}}_{2,\infty} \cap L^{\infty}} + \mu^{-1}\|u_0\|_{\dot{B}^{\frac{N}{2}-1}_{2,1}} 
%\leq c.
%\end{equation*}
%That result has been extended  to more general

In all those works however, the regularity requirements for the density are much too strong to consider piecewise constant functions. 
That difficulty has been by-passed  in a joint work of the first author with {\scshape P.B. Mucha} \cite{DanM2012}, where  well-posedness 
has been established in a critical regularity framework that allows for initial densities
that are  discontinuous along a $\cC^1$ interface (see the comments below Theorem \ref{thm:wellposed_INS}).
\smallbreak

Before writing out  the statement we are referring to and giving 
the main results of the present paper,   we need to introduce some notations. In all the paper, we agree that $A \lesssim B$ means  $A \leq C B$ for some harmless ``constant'' $C,$
the meaning of which may be guessed from the context. 
For $T\in]0,+\infty[,$ $p\in[1,+\infty]$ and $E$ a Banach space, the notation $L^p_T(E)$ 
designates the space of $L^p$ functions on $]0,T[$ with values in $E,$ and $L^p(\R_+; E)$ 
corresponds to the case $T=+\infty.$ For simplicity, we keep  the same notation 
for  vector or matrix-valued functions. 
\medbreak
Next, let us  recall  the definition of Besov spaces (following  e.g. \cite[Section 2.2]{BCD2011}). To this end, consider  two  smooth radial functions $\chi$ and $\varphi$ supported in  $\{ \xi \in \R^N :|\xi| \leq {4}/{3} \}$ and $\{ \xi \in \R^N : {3}/{4} \leq |\xi| \leq {8}/{3} \},$  respectively, and satisfying
\begin{align} \label{eq:L-P_decomp}
 \sum_{j \in \mathbb{Z} } \varphi(2^{-j}\xi) =1,~~~ \forall ~\xi \in \R^N \backslash \{0\},\quad
  \chi(\xi) + \sum_{j \geq 0 } \varphi(2^{-j}\xi) = 1,~~~\forall ~\xi \in \R^N .
\end{align}

Next, let us introduce the following Fourier truncation operators:
\begin{equation*}
\hDelta_j := \varphi(2^{-j}D),\,\, \hS_j := \chi(2^{-j} D),\,\,\forall j\in \Z; \qquad\Delta_j := \varphi(2^{-j}D), \,\,\forall j \geq 0, \quad\Delta_{-1} : = \chi( D).
\end{equation*}
For all triplet  $(s,p,r) \in \R \times [1,\infty]^2,$
the homogeneous Besov space  $\hB^s_{p,r}(\R^N)$ (just denoted by  $\hB^s_{p,r}$ if the
value of the dimension is clear from the context)  is defined~by
\begin{equation}\label{def:besov}
\hB^{s}_{p,r}(\R^N):= \big\{ \,u \in \mathcal{S}'_h(\R^N): \|u\|_{\hB^s_{p,r}} :=\big\|2^{js}\|\hDelta_j u \|_{L^p} \big\|_{\ell^r(\mathbb{Z}) } <\infty \,\big\},
\end{equation}
  where $\mathcal{S}'_h(\R^N)$ is  the  subspace of tempered distributions $\mathcal{S}'(\R^N)$ defined by
\begin{equation*}
\mathcal{S}'_h(\R^N):=\bigl\{ \,u \in \mathcal{S}'(\R^N) :  \lim_{j\rightarrow -\infty} \hS_j  u = 0 \,\bigr\}\cdotp 
\end{equation*}
We shall also use sometimes  the following inhomogeneous Besov spaces:
\begin{equation}\label{def:inhbesov}
B^{s}_{p,r}(\R^N):= \big\{ \,u \in \mathcal{S}'(\R^N): \|u\|_{B^s_{p,r}} :=\big\|2^{js}\|\Delta_j u \|_{L^p} \big\|_{\ell^r(\N \cup \{-1\}) } <\infty \,\big\}.
\end{equation}
Throughout the paper, we agree that the  notation $b^s_{p,r}(\R^N)$ designates
both  $B^{s}_{p,r}(\R^N)$ and $\hB^{s}_{p,r}(\R^N).$  
\smallbreak
It is  well-known  that the family of Besov spaces contains more classical items like the Sobolev or H\"older spaces.
For instance  $\dot B^s_{2,2}(\R^N)$  coincides with the homogeneous  Sobolev space $\dot H^s(\R^N)$ 
and we have \begin{equation}\label{spc:Holder}
B^s_{\infty,\infty}(\R^N) = \cC^{0,s}(\R^N) = L^{\infty}(\R^N) \cap \hB^s_{\infty,\infty}(\R^N)\quad \mbox{if} \quad s\in ]0,1[.
\end{equation}
To emphasize that latter connection between H\"older 
 and Besov spaces, we shall often   use the notation
$\hsC^{s}:= \hB^{s}_{\infty,\infty}$ (or $\sC^{s}:= B^{s}_{\infty,\infty}$) for \emph{any} $s\in \R.$  
\medbreak
When   investigating evolutionary equations in critical Besov spaces,   it is wise  to use
the following \emph{tilde homogeneous Besov spaces} first introduced by {\scshape Chemin}  in \cite{Che1999}:
 for any $t\in ]0,+\infty]$ and $(s,p,r,\gamma) \in \R \times [1,+\infty]^3$, we set
\begin{equation*}
\widetilde{L}^\gamma_t \big(\hB^{s}_{p,r} \big):= \Bigl\{ u \in \mathcal{S}'(]0,t[ \times \R^N):  \lim_{j \rightarrow -\infty} \hS_j u =0  \quad \mbox{in} \quad L^\gamma_t(L^\infty) \quad \mbox{and} \quad\|u\|_{\widetilde{L}^\gamma_t (\hB^{s}_{p,r})} <\infty\Bigr\},
\end{equation*}
where 
\begin{equation*}
\|u\|_{\widetilde{L}^\gamma_t (\hB^{s}_{p,r})} :=\big\|2^{js}\|\hDelta_j u\|_{L^\gamma_t(L^p)} \big\|_{\ell^r(\Z)} <\infty .
\end{equation*}
The index $t$ will be omitted if it is equal to $+\infty,$ and we shall denote
$$
\wt\cC_b(\R_+;\dot B^s_{p,r}):=\wt L^\infty(\R_+;\dot B^s_{p,r})\cap \cC(\R_+;\dot B^s_{p,r}).
$$
Finally, we shall make use of \emph{multiplier spaces} associated to couples  $(E, F)$ of  Banach spaces
 included in the set of  tempered distributions.   The definition goes as follows:
\begin{Def}
Let $E$ and $F$ be two Banach spaces embedded in $\cS'(\R^N).$
The \emph{multiplier space}  $\cM(E \to F)$ (simply denoted by $\cM(E)$ if $E=F$)
is  the set  of those functions $\varphi$ satisfying $\varphi u\in F$ for all $u$ in $E$ and, 
additionally,
\begin{equation}\label{eq:def_norm_multiplier}
\|\varphi\|_{\cM(E \to F)} := \sup_{ \substack{u \in E \\ \|u\|_{_E} \leq 1} } \|\varphi u\|_F<\infty.
\end{equation}
\end{Def}
It goes without saying that $\|\cdot\|_{\cM(E\to F)}$ is a norm on $\cM(E\to F)$
and that one may restrict the supremum in \eqref{eq:def_norm_multiplier}
to any \emph{dense} subset of $E.$
\medbreak

The following result  that has been proved in \cite{DanM2012}
is the starting point of our analysis. 
\begin{thm}\label{thm:wellposed_INS}
Let $p\in [1,2N[$ and $u_0$ be a divergence-free vector field with coefficients in $\hB^{\frac{N}{p}-1}_{p,1}$. Assume that  $\rho_0$ belongs to the multiplier space $\cM \big(\hB^{\frac{N}{p}-1}_{p,1}\big).$ There exist two  constants $c$ and 
$C$ depending only on $p$ and on $N$ such that if 
\begin{equation*}
\|\rho_0 -1 \|_{\cM \big(\hB^{\frac{N}{p}-1}_{p,1}\big)} +  \mu^{-1}\|u_0\|_{\hB^{\frac{N}{p}-1}_{p,1}} \leq c
\end{equation*}
then System \eqref{eq:INSmu} in $\R^N$ with $N\geq2$ has a unique solution $(\rho, u, \nabla P)$ satisfying
 \begin{itemize}
\item  $\rho \in L^{\infty}\Big(\R_+; \cM \big( \hB^{\frac{N}{p}-1}_{p,1} \big)\Big),$
\item $u\in \wt\cC_b \big(\R_+;\dot B^{\frac Np-1}_{p,1}\big),$
\item  $(\d_tu,\nabla^2 u,\nabla P) \in L^1 \big(\R_+;\dot B^{\frac Np-1}_{p,1} \big).$ 
 \end{itemize}
 Furthermore, 
the following inequality is fulfilled:
\begin{equation}\label{est:up}
\|u\|_{\wt L^\infty\big(\R^+;\dot B^{\frac Np-1}_{p,1} \big)}+\|\d_tu,\mu\nabla^2u,\nabla P\|_{L^1 \big(\R^+;\dot B^{\frac Np-1}_{p,1} \big)} \leq C\|u_0\|_{\dot B^{\frac Np-1}_{p,1}}.
\end{equation}
\end{thm}
A similar result (only local in time) may be proved for large $u_0.$ 
However the smallness condition on $\rho_0-1$ is still needed,
and whether one can extend Theorem \ref{thm:wellposed_INS} to the case of   \emph{large} density variations and \emph{critical} velocity fields is totally open. 
\smallbreak
By classical embedding, having $\nabla^2u$ in $L^1(\R_+;\dot B^{\frac Np-1}_{p,1})$ implies that $\nabla u$ is 
in $L^1(\R_+;\cC_b).$ Therefore the flow $\psi$ of $u$ is in $\cC^1.$
Now, it has been observed in  \cite{DanM2012} that  for any uniformly $\cC^1$ bounded domain $\cD_0,$
 the function $\mathds{1}_{\cD_0}$  belongs to  
$\cM \big(\hB^{s}_{p,1}\big)$  whenever $-1+\frac1p<s<\frac1p\cdotp$
Therefore, one may deduce from Theorem \ref{thm:wellposed_INS} that  if
 $\rho_0$ is   given by \eqref{eq:rho0}, if $u_0$ is in $\dot B^{\frac Np-1}_{p,1}$ for some $N-1<p<2N$
 and if  
$$\|u_0\|_{\dot B^{\frac Np-1}_{p,1}}+|\eta_2-\eta_1|\quad\hbox{is  small enough}$$
then System \eqref{eq:INSmu} admits a unique global solution in the above regularity 
class with    $\rho(t,\cdot)$  given by \eqref{eq:rhot} and   $\cD_t=\psi(t,\cD_0)$
 in  $\cC^1$ for all time $t\geq0.$  
\medbreak
The present paper aims at  propagating $\cC^{1,\ep}$ regularity of density patches
for any $\ep \in ]0,1[$  and  \emph{within  a critical regularity framework.} 
 For simplicity,  we shall focus  on simply connected bounded domains $\cD_0,$ 
 and  $\cC^{1,\ep}$ regularity thus means that    there exists some open neighborhood
$V_0$ of $\cD_0$ and a function $f_0:\R^N\to\R$ of class $\cC^{1,\ep}$ such that 
\begin{equation}\label{eq:f0}
\cD_0=f_0^{-1}(\{0\})\cap V_0\quad\hbox{and }\ 
\nabla f_0\ \hbox{ does not vanish on }\ V_0.
\end{equation}
As the viscosity coefficient $\mu$ will be fixed once and for all, we shall set it to $1$ 
for notational simplicity.  Likewise,  we shall assume the
reference density at infinity to be $1.$ 
\smallbreak
 Our main statement of propagation 
of H\"older regularity of density patches for $(INS)$ in the plane reads as follows.
\begin{thm}\label{thm:N2} Let  $\cD_0$ be  a simply connected bounded domain of $\R^2$ satisfying  \eqref{eq:f0} for some $\ep$ in $]0,1[.$
There exists a constant $\eta_0$ depending only on $\cD_0$ so that for all $\eta\in]-\eta_0,\eta_0[$
if the initial density is given by 
\begin{equation}\label{eq:density_patch}
\rho_0 := (1+\eta) \mathds{1}_{\mathcal{D}_0} + \mathds{1}_{\mathcal{D}^c_0},
\end{equation}  and   the divergence free vector-field $u_0\in L^2$ 
has vorticity  $\omega_0:=\d_1u_0^2-\d_2u_0^1$ with zero average and such that
\begin{equation}\label{eq:vorticity}
\omega_0=\wt\omega_0\, \mathds{1}_{\mathcal{D}_0}
\end{equation}
for some small enough function  $\wt\omega_0$ with H\"older regularity, %in $\cC^{0,\ep},$
then   System $(INS)$ has  a unique solution $(\rho, u,\nabla P)$  with the properties
listed in  Theorem \ref{thm:wellposed_INS} for some suitable $p\in]1,4[$.

 In addition, if we denote by $\psi$ the flow of $u$ then for all $t\geq0,$ we have
\begin{equation}\label{eq:density_patcht}
\rho(t,\cdot) := (1+\eta) \mathds{1}_{\mathcal{D}_t} + \mathds{1}_{\mathcal{D}^c_t}
\ \hbox{ with  }\ \mathcal{D}_t := \psi (t,\mathcal{D}_0), 
\end{equation} 
and $\mathcal{D}_t$ remains a simply connected bounded domain of class $\cC^{1,\ep}.$ 
\end{thm} 
\begin{rmk} We need  the initial vorticity to be mean free, in order
to guarantee that $u_0$ belongs to some homogeneous Besov space
$\dot B^{\frac2p-1}_{p,1}.$ 
It is no longer needed in dimension $3$ (see Theorem \ref{thm:N3} below).

 Of course, there are many other examples of initial velocities  for which 
propagation  of $\cC^{1,\ep}$ patches holds true~: an obvious one is when 
$u_0$ has critical regularity and vanishes on a neighborhood of~$\cD_0.$
 \end{rmk}
 \begin{rmk}
 Our method would allow us to consider  large initial vorticities as in \eqref{eq:vorticity}. 
 However, we would end up with a local-in-time result only. 
 \end{rmk}
As in \cite{LiaoZh2016a,LiaoZh2016b}, Theorem \ref{thm:N2} will come up as a consequence 
of a much more general result of persistence of geometric structures for $(INS).$
To give the exact statement, we need to introduce 
for $(\sigma, p, T) \in \R \times [1,\infty] \times ]0,\infty] ,$  
  the space
\begin{equation*}
 \hE^\sigma_p (T):=\big\{ (v, \nabla Q) : v\in \wt\cC_b \big([0,T[; \hB^{\frac{N}{p}-1+\sigma}_{p,1}), 
\: (\d_t v, \nabla^2 v , \nabla Q ) \in L^1_T \big( \hB^{\frac{N}{p}-1+\sigma}_{p,1}\big)  \big\} ,
\end{equation*}
endowed with the norm 
\begin{equation*}
\|(v,\nabla Q)\|_{\hE^\sigma_p (T)} := \|v\|_{\wt{L}^\infty_T\big( \hB^{\frac{N}{p}+\sigma-1}_{p,1}\big)} + \|(\d_t v, \nabla^2 v, \nabla Q )\|_{L^1_T \big(\hB^{\frac{N}{p}+\sigma-1}_{p,1}\big)}.
\end{equation*}
 %The subspace  $\hsE^\sigma_p(T)$  of $\hE^\sigma_p(T)$ is defined by:
%\begin{equation*}
%\hsE^\sigma_p (T) := \big\{ (v, \nabla Q) \in \hE^\sigma_p(T):  v \in \wt{L}^{\infty}_T(   \hB^{\frac{N}{p}-1+\sigma}_{p,1})  \big\},
%\end{equation*}
%and is endowed with the norm
%\begin{equation*}
%\|(v,\nabla Q)\|_{\hsE^\sigma_p (T)} := \|v\|_{\wt{L}^\infty_T\big( \hB^{\frac{N}{p}+\sigma-1}_{p,1}\big) }  + \|(\d_t v, \nabla^2 v, \nabla Q )\|_{L^1_T \big(\hB^{\frac{N}{p}+\sigma-1}_{p,1}\big)}.
%\end{equation*}
For notational simplicity, we shall omit $\sigma$ or  $T$  in  the notation $\hE^\sigma_p(T)$  whenever $\sigma$ is zero or $T=\infty.$ For instance, $\hE_p := \hE^0_p (\infty).$ 
\begin{thm}\label{thm:main_dINS}
Let $\ep$ be in $]0,1[$ and $p$ satisfy
\begin{equation}\label{cdt:p}
\frac N2<p < \min \Big\{ \frac{N}{1-\ep}, 2N\Big\}\cdotp
\end{equation}
Let  $u_0$ be a divergence-free vector field with coefficients in $\hB^{\frac{N}{p}-1}_{p,1}$. Assume that the initial density $\rho_0$ is bounded and belongs to the multiplier space $\cM \big(\hB^{\frac{N}{p}-1}_{p,1}\big) \cap \cM \big(\hB^{\frac{N}{p}+\ep-2}_{p,1}\big).$ There exists a  constant $c$ depending only on $p$ and on $N$ such that if 
\begin{equation}\label{cdt:u0_small}
\|\rho_0 -1 \|_{\cM \big(\hB^{\frac{N}{p}-1}_{p,1}\big) \cap \cM \big(\hB^{\frac{N}{p}+\ep-2}_{p,1}\big)\cap L^\infty} + \|u_0\|_{{\hB}^{\frac{N}{p}-1}_{p,1}} \leq c,
\end{equation}
then System $(INS)$  has a unique solution $(\rho, u, \nabla P)$ with 
 $$\rho\in L^\infty\Big(\R_+ ; L^\infty\cap \cM \big(\hB^{\frac{N}{p}-1}_{p,1}\big) \cap \cM \big(\hB^{\frac{N}{p}+\ep-2}_{p,1}\big) \Big) \quad\hbox{and}\quad  (u,\nabla P)\in \hE_p.$$
 Moreover,  for any vector field $X_0$ with $\cC^{0,\ep}$ regularity 
 (assuming in addition that   $\ep>2-\frac Np$ if $\div X_0\not\equiv0$),  
 if the following striated-type conditions are fulfilled
$$\d_{X_0} \rho_0 \in  \cM \big(\hB^{\frac{N}{p}-1}_{p,1}\rightarrow \hB^{\frac{N}{p}+\ep-2}_{p,1}\big) \quad\hbox{and}\quad
\d_{X_0} u_0\in\hB^{\frac{N}{p}+\ep-2}_{p,1},$$ 
then System \eqref{eq:X} in $\R^N$ has a unique global 
 solution $X\in C_w(\R_+;\cC^{0,\ep}),$ and we have
\begin{equation*}
 \d_X \rho \in L^\infty \Big(\R_+ ; \cM \big(\hB^{\frac{N}{p}-1}_{p,1}\rightarrow \hB^{\frac{N}{p}+\ep-2}_{p,1}\big) \Big)\quad\hbox{and}\quad
 (\d_X u, \d_X\nabla P)\in \hE^{\ep-1}_p.
\end{equation*} 
\end{thm}
Some comments are in order:
\begin{itemize}
\item The divergence-free property on $X_0$ is  conserved during the evolution because if one takes  the divergence of \eqref{eq:X},  and remember that $\div u =0,$ then we get
\begin{equation}\label{eq:divX}
\left\{
\begin{array}{l}
 \d_t \div X + u\cdot \nabla \div X = 0,\\[1ex]
 \div X |_{t=0}=\div X_0. 
\end{array}
\right.
\end{equation}
\item In the case $\div X_0 \not= 0,$ the additional constraint on $(\ep,p)$ 
is due to the fact that the product of a general $\cC^{0,\ep}$ function with a $\dot B^{\frac Np-2}_{p,1}$
distribution need not be defined if the sum of regularity coefficients, namely $\ep+\frac Np-2,$ is negative.
\item The vector field $X$ given by \eqref{eq:X} has
the Finite Propagation Speed  Property. 
Indeed, from the definitions of the flow and of the space $\dot E_p,$ and from the embedding of
$\dot B^{\frac Np}_{p,1}(\R^N)$ in $\cC_b(\R^N),$ we readily get for all $t\geq0$ and $x\in\R^N,$
$$
\bigl|\psi(t,x)-x\bigr|\lesssim \sqrt t\,\|u\|_{L^2_t(\dot B^{\frac Np}_{p,1})}\leq C\sqrt t\,\|u_0\|_{\dot B^{\frac Np-1}_{p,1}}.$$
Therefore, if the initial vector field $X_0$ is supported in the  set $K_0$
then $X(t)$ is supported in some set $K_t$ such that  
 \begin{equation*}
 {\rm diam}(K_t) \leq  {\rm diam}(K_0) +  C\sqrt t\,\|u_0\|_{\dot B^{\frac Np-1}_{p,1}}.
 \end{equation*}
\item 
One can prove a similar result (only local in time even in the 2D case) if we remove the smallness assumption on $u_0.$ Moreover, we expect  our method to be appropriate for handling H\"older regularity $\cC^{k,\ep}$ if making suitable assumptions on $\d_{X_0}^j\rho_0$ and $\d_{X_0}^j u_0$ for $j=0,\cdots,k.$ We refrained from writing out here this generalization to keep the presentation as elementary as possible.
\end{itemize}
\medbreak
We end this section with a short presentation of the main ideas of the proof of Theorem \ref{thm:main_dINS}.
The starting point  is  Theorem \ref{thm:wellposed_INS} that provides us with 
a global solution $(\rho,u,\nabla P)$ with $\rho\in L^\infty\Big(\R_+ ; \cM \big(\hB^{\frac{N}{p}-1}_{p,1}\big)\Bigr)$ and $(u,\nabla P)\in\hE_p.$
The flow $\psi$ of $u$ is thus in $\cC^1.$
Our main task  is to prove that $X(t,\cdot)$ remains in $\cC^{0,\ep}$ for all time. 
Now, \eqref{eq:X} ensures that 
\begin{equation*}
X(t,x)= X_0\big(\psi^{-1}_t(x)\big) + \int^t_0 \d_X u \Big(t', \psi_{t'} \big(\psi^{-1}_t(x) \big)\Big)dt'.
\end{equation*}
Because $\psi_t$ is a $\cC^1$ diffeomorphism of $\R^N,$ it  thus suffices 
to show that $\d_Xu$ is in  $L^1_{loc}(\R_+;\cC^{0,\ep}).$ 
    \medbreak
  Note that Equation  \eqref{eq:X} exactly states that $[D_t, \d_X]=0,$ where  $D_t : = \d_t + u \cdot \nabla$ stands for  the material derivative associated to~$u.$  Therefore differentiating  the mass and momentum  equations of $(INS)$ along $X,$ we discover that 
  \begin{equation}\label{eq:dXrho}
D_t\d_X\rho=0
\end{equation}
and that
  \begin{equation}\label{eq:dXu}
\rho D_t \d_X u + \d_X \rho D_t u- \d_X \Delta u + \d_X \nabla P =0.
\end{equation}
On the one hand, Equation \eqref{eq:dXrho} implies that  any (reasonable) regularity assumption 
for $\rho$ along $X$ is  conserved through the evolution. 
On the other hand, even though \eqref{eq:dXu} has some similarities with  the Stokes system, 
  it is not  clear that it does have  the same  smoothing properties, as its
  coefficients have  very low regularity. 
One of the difficulties lies in the product of 
the  discontinuous function $\rho$ with  $D_t\d_Xu,$ as having
only $\d_X u$ in $\cC^{0,\ep}$ suggests that $D_t\d_Xu$ has \emph{negative} regularity.

Our strategy is to  assume that $\rho$ belongs to some 
 multiplier space  corresponding  to the  space
to which $D_t\d_X u$ is expected to belong. As our flow is $\cC^1,$ 
propagating multiplier informations turns out to be rather straightforward (see Lemma \ref{lemma:multiplier}). 
 Thanks to this new
 viewpoint, one can  avoid using the tricky energy estimates
and iterated differentiation along vector fields (requiring higher regularity of the patch) 
that were the cornerstone of the work by {\scshape Liao and Zhang}.
In fact,  \emph{under the smallness assumption \eqref{cdt:u0_small}}
which, unfortunately, forces the fluid to have small density variations, we succeed 
in closing the estimates via only one differentiation along $X.$
 This makes the proof rather elementary and allows us to propagate low H\"older regularity. 

However, even with the above viewpoint, whether one can differentiate 
 terms like $\Delta u$ or $\nabla P$
along $X$ within our critical regularity framework is not totally clear. 
In fact,  as in our recent work \cite{DanZhx2016} 
dedicated to the incompressible Boussinesq system,
we shall   resort to elementary paradifferential calculus
(first introduced by {\scshape Bony} in \cite{Bo1981}).

 Let us briefly recall how it works. Fix some suitably large integer $N_0$ and 
introduce the following \emph{paraproduct} and \emph{remainder} operators:
$$\hT_u v := \sum_{j\in \Z} \hS_{j-N_0}u \hDelta_j v\quad\hbox{and}\quad 
\hR(u,v) \equiv \sum_{j\in\Z} \hDelta_{j} u \thDelta_{j} v := \sum_{\genfrac{}{}{0pt}{}{j\in\Z}{|j-k|\leq N_0}} \hDelta_{j} u \hDelta_{k} v.
$$
It is clear that, formally, any product may be decomposed as follows:
\begin{equation}\label{eq:bony}
uv = \hT_u v + \hT_v u + \hR(u,v).
\end{equation} 
To overcome the problem with the definition (and estimates) of $\d_X\Delta u$ and $\d_X\nabla P,$ 
the idea is to  change  the vector field $X$ to the 
para-vector field operator $\hcT_X \cdot := \dot{T}_{X^k} \d_k \cdot.$ 
This is justified because in our regularity framework  $\hcT_X$ turns out to be 
 \emph{the principal part} of operator $\d_X.$  Typically, 
  $X$ will act on $\nabla P$ or on $\Delta u$ which are in $L^1 \big(\R_+;\dot B^{\frac Np-1}_{p,1}(\R^N)\big).$
Now, suppose that $(X,f) \in \big(\hsC^\ep (\R^N)\big)^N \times \hB^{\frac{N}{p}-1}_{p,1}(\R^N)$  with $(\ep,p)\in]0,1[\times[1,+\infty]$ such that   
\begin{equation}\label{cdt:p_X}
\frac Np \in \,]\,1-\ep,\,2 \,[\ \hbox{ if }\ \div X=0, \quad\hbox{and }\ \frac Np \in \,]\,2-\ep,\,2\,[
\ \hbox{ otherwise.}
\end{equation} 
Then, by virtue of  Bony's decomposition \eqref{eq:bony}, we have
\begin{equation*}
( \hcT_X -\d_X ) f = \hT_{\d_k f} X^k + \d_k \hR(f, X^k) - \hR(f,\div X).
\end{equation*} 
Taking advantage of  classical continuity results for operators $\hT$ and $\hR$ (see \cite{BCD2011}),
we discover that, under Condition \eqref{cdt:p_X}, we have 
\begin{equation}\label{es:dXf_TXf}
\| ( \hcT_X -\d_X )f \|_{\hB^{\frac{N}{p}+\ep-2}_{p,1}} \lesssim  \|f\|_{\hB^{\frac{N}{p}-1}_{p,1}} \|X\|_{\hsC^{\ep}}.
\end{equation}

 Now, incising the term $\d_X u$ by the scalpel  $\hcT_X $ in \eqref{eq:dXu} and  applying $\hcT_X $ to the third  equation of $(INS)$ yield
\begin{equation} \label{eq:dXINS}
\left\{\begin{array}{l}
\rho D_t \hcT_X u-\Delta \hcT_X u + \nabla \hcT_X P = g,\\
\div \hcT_X u = \div (\hT_{\d_k X} u^k -\hT_{\div X} u),\\
\hcT_X u|_{t=0} =  \hcT_{X_0} u_0
\end{array}\right.
\end{equation}
with
  \begin{multline}\label{eq:g_dINS}
  g:= - \rho [\hcT_{X}, D_t]u +[\hcT_X,\Delta] u - [\hcT_X, \nabla] P + (\d_X-\hcT_X) (\Delta u - \nabla P) 
 \\ -\d_X \rho D_t u + \rho (\hcT_X-\d_X)D_t u.
  \end{multline}
  This  surgery leading to  \eqref{eq:dXINS} is quite effective for three reasons.
   Firstly, all the commutator terms in \eqref{eq:g_dINS} are under control (see the Appendix). More importantly, 
   as $D_t\d_Xu$ and $D_t\hcT_Xu$ are in the same Besov space, 
   we can still use the multiplier type regularity on the density that we  pointed out before. Lastly, Condition \eqref{cdt:u0_small} ensures that $(\hcT_X -\d_X)u$ is indeed
   a (small) remainder term.   

Of course,  the divergence free condition need not be satisfied by $\hcT_X u.$
  We shall thus further modify the above  Stokes-like equation  
  so as to enter in the standard  maximal regularity theory.
  Then, under the smallness condition \eqref{cdt:u0_small}, one can 
   close the estimates involving striated regularity along $X,$ globally in time.   
   \medbreak
The rest of the paper unfolds as follows. 
In the next section, we show that Theorem \ref{thm:main_dINS}  entails a general (but not
so explicit) result of  persistence of H\"older regularity for patches of density
in any dimension, under suitable striated regularity assumptions for the velocity. 
We shall then obtain  Theorem \ref{thm:N2}, and an analogous result
in dimension $N=3$. 
Section \ref{sect:3} is devoted to the proof of   all-time
persistence of striated regularity (that is Theorem \ref{thm:main_dINS}). 
Some technical results pertaining to commutators
and  multiplier spaces are postponed in appendix.

%%%%%%%%%%%%%%%%%%%%%%%%%%%%%%%%%%%%%%%%%%%%%%%%%%%%%%%%%%%%%%%%%%%%%%%%%%%%%%%%%%%%

%%%%%%%%%%%%%%%%%%%%%%%%%%%%%%%%%%%%%%%%%%%%%%%%%%%%%%%%%%%%%%%%%%%%%%%%%%%%%%%%%%%
 
\section{The density patch problem}

This section is devoted to the proof of results of persistence  of regularity for  patches of constant densities,  
taking Theorem  \ref{thm:main_dINS} for granted. 
Throughout this section we shall use repeatedly the fact (proved in e.g. see \cite[Lemma A.7]{DanM2012}) that for any  (not necessarily bounded) domain $\cD$ of  $\R^N$ with uniform $\cC^{1}$ boundary, we have 
\begin{equation*}
 \mathds{1}_{\mathcal{D}} \in \cM\big(\hB^{s}_{p,r} (\R^N)\big) \quad \mbox{whenever} \quad (s,p,r) \in ]\frac{1}{p}-1,\frac{1}{p}[ \times ]1,\infty[\times [1,\infty].
\end{equation*}  
{}From that property, we deduce that  if $(\ep, p) \in ]0,1[ \times ]N-1, \frac{N-1}{1-\ep}[ ,  $ then the density 
$\rho_0$ given by \eqref{eq:density_patch}
belongs to $\cM\big(\hB^{\frac{N}{p}-1}_{p,r}(\R^N)\big) \cap \cM\big(\hB^{\frac{N}{p}+\ep-2}_{p,r}(\R^N)\big).$ 
\smallbreak

As a start, let us give a result of persistence of regularity, 
under rather general hypotheses.
\begin{prop}\label{prop:NN}  Assume that $\rho_0$ is given by \eqref{eq:density_patch} 
with small enough $\eta$ and 
some domain $\cD_0$ satisfying \eqref{eq:f0}.
Let $u_0$ be a small enough divergence free vector field with coefficients in 
$\dot B^{\frac Np-1}_{p,1}$ for some  $N-1<p<\min\big\{\frac{N-1}{1-\ep},2N\big\}\cdotp$
Consider a family $(X_{\lambda,0})_{\lambda\in\Lambda}$ 
of $\cC^{0,\ep}$ divergence free vector fields tangent
to $\cD_0$ and such that $\d_{X_{\lambda,0}}u_0 \in \dot B^{\frac Np+\ep-2}_{p,1}$ for 
all $\lambda\in\Lambda.$
\smallbreak
Then the  unique solution $(\rho,u,\nabla P)$ of $(INS)$ given by  Theorem \ref{thm:wellposed_INS}
satisfies  the following additional properties:
\begin{itemize}
\item  $\rho(t,\cdot)$ is given by \eqref{eq:density_patcht}, 
\item  all the time-dependent vector fields $X_{\lambda}$ solutions to \eqref{eq:X} with initial data 
$X_{\lambda,0}$ are in $L^\infty_{loc}(\R_+;\cC^{0,\ep})$ and remain tangent to the patch for all time.
\end{itemize}
\end{prop}
\begin{proof}
As pointed out at the beginning of this section, our assumptions on $p$ ensure that
$\rho_0$ is in $\cM(\hB^{\frac Np-1}_{p,1}) \cap \big(\hB^{\frac Np-2+\ep}_{p,1}),$ and 
 \eqref{cdt:u0_small} is fulfilled if $\eta$ and $u_0$ are small enough.
Of course, $\d_{X_{\lambda, 0}}\rho_0\equiv0$ for all $\lambda\in\Lambda$ because 
the vector fields $X_{\lambda,0}$ are tangent to the boundary. 
Therefore, applying Theorem  \ref{thm:main_dINS} ensures that 
all the vector fields  $X_{\lambda}$ are  in $L^\infty_{loc}(\R_+;\cC^{0,\ep}).$
Now, if we consider a  level set function $f_0$ in $\cC^{1,\ep}$ associated to $\cD_0$ (see \eqref{eq:f0})
then the function $f_t:=f_0\circ\psi_t$ is associated to the transported domain 
$\cD_t=\psi_t(\cD_0),$ and easy computations show that
\begin{equation}\label{eq:Df}
D_t \nabla f=-\nabla u\cdot\nabla f\quad\hbox{with }\ (\nabla u)_{ij}=\d_iu^j.
\end{equation}
Therefore, as $X_\lambda$ satisfies \eqref{eq:X}, we have
$$D_t(X_\lambda\cdot\nabla f)=(D_tX_\lambda)\cdot\nabla f+X_\lambda\cdot(D_t\nabla f)=0,$$
which ensures that $X_\lambda$ remains tangent to the patch for all time.
\end{proof}

%%%%%%%%%%%%%%%%%%%%%%%%%%%%%%%%%%%%%%%%%%%%%%%%%%%%%%%%%%%%%%%%%%%%%%%%%%%%%%%%%%%%%%%

\subsection{The two-dimensional case}

Here we prove   Theorem \ref{thm:N2}.  
So we assume that $\omega_0=\wt\omega_0\mathds{1}_{\mathcal{D}_0}$
for some small enough function $\wt\omega_0$ 
that can be taken compactly supported and in the nonhomogeneous Besov space  $B^{\alpha}_{\infty,1}(\R^2)$
for some $\alpha\in]0,\ep[,$   with no loss of generality. As we assumed that $u_0$ has some
decay at infinity, it may be computed from $\omega_0$ through the following 
 Biot-Savart law: 
$$
u_0=(-\Delta)^{-1}\nabla^\perp\omega_0.
$$
We claim that $u_0$ belongs to all spaces $\dot B^{\frac2p-1}_{p,1}(\R^2)$
with $p>1.$  Indeed, let us write that
$$
u_0=\dot S_0 u_0+({\rm Id}-\dot S_0)u_0.
$$
Because $\omega_0$ is bounded, compactly supported and mean free, 
it is obvious that $\dot S_0 \omega_0$ is smooth, in all 
Lebesgue spaces and also mean free.   Biot-Savart law thus ensures that $\dot S_0 u_0$ belongs
to all Lebesgue spaces $L^q$ with $q>1$ (as it is smooth and behaves like ${\mathcal O}(|x|^{-2})$
at infinity, due to the mean free property, see e.g. \cite[p. 92]{MB02}).  Hence for any  $1<q<2$ and $p\geq q,$ one may write
$$
\|\dot S_0 u_0\|_{\dot B^{\frac 2p-1}_{p,1}}\lesssim \|\dot S_0u_0\|_{\dot B^{\frac 2p-\frac2q}_{p,\infty}}
\lesssim \|\dot S_0 u_0\|_{L^q}\leq C_{\omega_0}.
$$
As regards the high frequency part of $u_0,$ because the Fourier multiplier $({\rm Id}-\dot S_0)\nabla^\bot
(-\Delta)^{-1}$ is homogeneous of degree $-1$ away from a neighborhood of $0,$ we have
$$\begin{aligned}
\|({\rm Id}-\dot S_0)u_0\|_{\dot B^{\frac 2p-1}_{p,1}}
&=\|({\rm Id}-\dot S_0)\nabla^\perp(-\Delta)^{-1}\omega_0\|_{\dot B^{\frac 2p-1}_{p,1}}\\
&\lesssim\|({\rm Id}-\dot S_0)\omega_0\|_{\dot B^{\frac2p-2}_{p,1}}
\lesssim \|({\rm Id}-\dot S_0)\omega_0\|_{L^p}\lesssim\|\omega_0\|_{L^1\cap L^\infty}.\end{aligned}
$$
Next, consider  the divergence free vector field
$X_0=\nabla^\perp f_0$ where $f_0$ is given by \eqref{eq:f0} and is (with no loss of generality)
compactly supported.  If it is true  that\begin{equation}\label{eq:dX0}
\d_{X_0} u_0\in \dot B^{\frac 2p-2+\ep}_{p,1}\quad\hbox{for some }\ 1<p<\min\Bigl(\frac1{1-\ep},4\Bigr),
\end{equation}
then one can apply  Proposition \ref{prop:NN} which ensures that the transported vector field $X_t$
remains in $\cC^{0,\ep}$ for all $t\geq0$. Now, it is classical that we have $X_t=(\nabla f_t)^\perp$
with $f_t=f_0\circ\psi_t.$ Hence $\cD_t$ has a $\cC^{1,\ep}$ boundary. 
\smallbreak
Let us  establish \eqref{eq:dX0}. 
Of course, by embedding, we have  $X_0$ in $B^\alpha_{\infty,1}.$  
Now,  \eqref{es:dXf_TXf} ensures that for any $p\geq1$ satisfying $\frac2p+\ep-1>0,$ 
\begin{equation}\label{eq:BS0}
\|\hcT_{X_0}u_0-\d_{X_0}u_0\|_{\dot B^{\frac 2p+\ep-2}_{p,1}}\lesssim \|u_0\|_{\dot B^{\frac 2p-1}_{p,1}}
\|X_0\|_{\hsC^\ep}.
\end{equation}
{}From Biot-Savart law, we get
$$
\hcT_{X_0}u_0=\hcT_{X_0}(-\Delta)^{-1}\nabla^\perp\omega_0
=(-\Delta)^{-1}\nabla^\perp\hcT_{X_0}\omega_0+[\hcT_{X_0},(-\Delta)^{-1}\nabla^\perp]\omega_0,
$$
whence  using  Lemma \ref{lemma:ce1},
\begin{equation}\label{eq:BS1}
\| \hcT_{X_0}u_0
-(-\Delta)^{-1}\nabla^\perp\hcT_{X_0}\omega_0\|_{\dot B^\alpha_{p,1}}
\lesssim \|X_0\|_{\dot B^\alpha_{\infty,1}}\|\omega_0\|_{L^p}.
\end{equation}
Next, we notice that
$$\hcT_{X_0}\omega_0-\div(X_0\omega_0)=-\div\bigl(\hT_{\omega_0}X_0+\hR(\omega_0,X_0)\bigr).
$$
Therefore, taking advantage of standard continuity results for $\hT$ and $\hR,$ 
we have 
\begin{equation}\label{eq:BS2}
\|\hcT_{X_0}\omega_0-\div(X_0\omega_0)\|_{\dot B^{\alpha-1}_{p,1}}\lesssim \|\omega_0\|_{L^p}
\|X_0\|_{\dot B^\alpha_{\infty,1}}\ \hbox{ for all }\ p\geq1.
\end{equation}
Finally, because $X_0$ and $\wt\omega_0$ are compactly supported and 
in $B^\alpha_{\infty,1},$  Proposition \ref{prop:besov-holder} and obvious embedding
ensure  that 
$$X_0\quad\hbox{and}\quad \wt\omega_0\ \hbox{ are in }\ \dot B^\alpha_{p,1}\cap L^\infty.$$
Hence, remembering that $\div({X_0}\omega_0)=\div(X_0\,\wt\omega_0\,\mathds{1}_{\mathcal{D}_0}),$ 
that $\div X_0=0$ and that $\d_{X_0}\mathds{1}_{\mathcal{D}_0}=0,$
Corollary \ref{cor:dXfg_3D}  implies that  $\div(X_0\omega_0)$ belongs to  $\dot B^{\alpha-1}_{p,1}.$
    
  Putting \eqref{eq:BS0}, \eqref{eq:BS1} and \eqref{eq:BS2} together, we conclude that
  \eqref{eq:dX0} is fulfilled provided the Lebesgue index $p$ defined by 
  \begin{equation}\label{eq:p2}
  \alpha =\frac2p-2+\ep\end{equation}
  is in  $]1,\min(4,\frac 1{1-\ep})[.$    As $0<\alpha<\ep,$ this  is indeed the case.
    This completes the proof of Theorem \ref{thm:N2}. \qed
 
%%%%%%%%%%%%%%%%%%%%%%%%%%%%%%%%%%%%%%%%%%%%%%%%%%%%%%%%%%%%%%%%%%%%

\subsection{The three-dimensional case}

As a second application of Proposition \ref{prop:NN}, we now want to generalize Theorem \ref{thm:N2}
to the three-dimensional case. Our result reads as follows.
\begin{thm}\label{thm:N3} 
Let $\cD_0$ be a $\cC^{1,\ep}$ simply connected bounded domain of $\R^3$ with 
$\ep\in]0,1[$ and  $\rho_0$ be given by  \eqref{eq:density_patch}
for some small enough $\eta.$  Assume that  the initial velocity $u_0$ has
coefficients in $\cS'_h(\R^3)$ and  
vorticity\footnote{For any point $Y \in \R^3$, we set $X \wedge Y := (X^2 Y^3-X^3Y^2,\, X^3 Y^1-X^1Y^3,\, X^1 Y^2-X^2Y^1 )$ where $X$ stands for an element of  $\R^3$ or  for the  $\nabla$ operator.}
\begin{equation*}
\Omega_0:=\nabla\wedge u_0=\wt{\Omega}_0 \mathds{1}_{\mathcal{D}_0},
\end{equation*}
for some small enough  $\wt{\Omega}_0$ in $ \cC^{0,\delta}(\R^3;\R^3)$ ($\delta  \in ]0,\ep[$) with  $\div \wt{\Omega}_0 =0$ and $\wt{\Omega}_0 \cdot \vec{n}_{_{\cD_0}} |_{\d \cD_0} \equiv 0$ (here $\vec{n}_{_{\cD_0}}$ denotes the outwards unit normal of the domain $\cD_0$).  

There exists a unique solution $(\rho, u,\nabla P)$ to System $(INS)$ with the properties
listed in  Theorem \ref{thm:wellposed_INS} for some suitable $p$ satisfying
\begin{equation}\label{eq:p3}
2<p<\min\biggl(\frac2{1-\ep}, 6\biggr)\cdotp
\end{equation} 
Furthermore,  for all $t\geq0,$ we have \eqref{eq:density_patcht}
 and $\mathcal{D}_t$ remains a simply connected bounded domain of class $\cC^{1,\ep}.$ 
\end{thm} 
\begin{proof} Without loss of generality, one may assume that $\wt\Omega_0$ is compactly supported
(as multiplying it by a cut-off function with value $1$ on $\cD_0$ will not change $\Omega_0$). 
Like in the 2D case, we first have check that $u_0$ satisfies the assumptions of Proposition \ref{prop:NN}. 
As it is divergence free and decays at infinity (recall that  $u_0\in\cS'_h$), it is  
given by the Biot-Savart law:
\begin{equation}\label{eq:BS3}
u_0=(-\Delta)^{-1}\nabla\wedge\Omega_0, \quad \hbox{with}\,\,\,\Omega_0 = \wt{\Omega}_0 \, \mathds{1}_{\mathcal{D}_0}.
\end{equation}
Let us first check that $u_0$ belongs to 
  $\hB^{\frac{3}{p}-1}_{p,1}$ for some $p$ satisfying Condition \eqref{eq:p3}. 
Recall that the characteristic function of any bounded domain with $\cC^1$ regularity 
belongs to all Besov spaces $B^{\frac1q}_{q,\infty}$ with $1\leq q\leq\infty$ (see e.g. \cite{Tri1992}). 
Hence  combining  Proposition \ref{prop:equiv_besov} and the embedding \eqref{es:nonhom_hom} gives
\begin{equation}\label{embd_D0}
 \mathds{1}_{\mathcal{D}_0} \in \ \cE' \cap b^{\frac{1}{q}}_{q,\infty} \hookrightarrow \hB^{\frac{3}{q}-2}_{q,1}, \quad \hbox{for any} \,\,\, q \in ]1,\infty[ \,\,\, \hbox{and}\,\, b \in \{B, \hB\}.
\end{equation}
Now, using Bony's decomposition and standard continuity results
for operators $\hR$ and $\hT,$ we discover that 
\begin{equation*}
\wt{\Omega}_0 \in \sC^{\delta}_c \hookrightarrow \cM \big(\hB^{\frac{3}{q}-2}_{q,1}\big) \quad \hbox{for any} \,\,\, q \in\Big ]\frac{3}{2},\, \frac{3}{2-\delta}\Big[.   
\end{equation*}
%Suppose $\psi \in \cC^{\infty}_c$ is some smooth cut-off function such that $\psi \equiv 1 $ on some neighborhood of $\cD_0.$ As a consequence of  \cite[Corollary 2.1.1]{DanM2015} and Lemma \ref{lemma:multiplier} (iii). we have
%\begin{equation*}\psi \, \wt{\Omega}_0 \in \cM \big(B^{\frac{3}{q}-2}_{q,1}\big) \hookrightarrow \cM \big(\hB^{\frac{3}{q}-2}_{q,1}\big) \quad \hbox{for any} \,\,\, q \in ]\frac{3}{2},\, \frac{3}{2-\delta}[. \end{equation*}
Hence the definition of Multiplier space and \eqref{embd_D0} yield
\begin{equation}\label{cdt:q0}
\Omega_0 =  \wt{\Omega}_0\,  \mathds{1}_{\mathcal{D}_0}\in \hB^{\frac{3}{q}-2}_{q,1} \quad \hbox{for any} \,\,\, q \in \Big]\frac{3}{2},\, \frac{3}{2-\delta}\Big[.
\end{equation} 
As $u_0$ is in $\cS'_h$ and  $(-\Delta^{-1})^{-1}\nabla\wedge$ in \eqref{eq:BS3} is a homogeneous multiplier of degree $-1,$  one can conclude that $$u_0 \in \hB^{\frac{3}{q}-1}_{q,1} \hookrightarrow \hB^{\frac{3}{p}-1}_{p,1}, \quad \mbox{for any $p\geq q$.}$$
Note  that for any value of $\delta$ in $]0,1[,$  one can find some  $p$ satisfying \eqref{eq:p3}.
\medbreak

Next, we consider some (compactly supported) level set function $f_0$ associated to $\d\cD_0$, and the three $\cC^{0,\ep}$ vector-fields $X_{k,0}:=e_k\wedge \nabla f_0$ with $(e_1,e_2,e_3)$ being the canonical basis of $\R^3.$ It is clear that those vector-fields are divergence free and tangent to $\d\cD_0.$  Let us check that we have $\d_{X_{k,0}}u_0\in\dot B^{\frac 3p-2+\ep}_{p,1}$ for some 
$p$ satisfying \eqref{eq:p3}.
 As in the two-dimensional case, this will follow from Biot-Savart law and the special structure of $\Omega_0.$ Indeed,  from  \eqref{es:dXf_TXf} and $\div X_{k,0}=0$, we have 
$$
\|\hcT_{X_{k,0}}u_0-\d_{X_{k,0}}u_0\|_{\dot B^{\frac 3p+\ep-2}_{p,1}}
\lesssim \|u_0\|_{\dot B^{\frac3p-1}_{p,1}}\|X_0\|_{\hsC^\ep},  \,\,\, \forall \,p \in \Big]\frac{3}{2}, \frac{3}{1-\ep}\Big[.
$$
Then \eqref{eq:BS3} yields 
$$
\hcT_{X_{k,0}}u_0=\hcT_{X_{k,0}}(-\Delta)^{-1}\nabla\wedge\Omega_0
=(-\Delta)^{-1}\nabla\wedge\hcT_{X_{k,0}}\Omega_0+[\hcT_{X_{k,0}},(-\Delta)^{-1}\nabla\wedge]\,\Omega_0.
$$
Thanks to Lemma \ref{lemma:ce1} and homogeneity of $(-\Delta^{-1})^{-1}\nabla\wedge$, it is thus sufficient to verify
that  $\hcT_{X_{k,0}}\Omega_0$ belongs to  $\hB^{\frac{3}{p}+\ep-3}_{p,1}$ for some $p$ satisfying \eqref{eq:p3}.  In fact, from the  decomposition 
$$\hcT_{X_{k,0}}\Omega_0-\div(X_{k,0}\Omega_0)
=-\div\bigl(\hT_{\Omega_0}X_{k,0}+\hR(\Omega_0,X_{k,0})\bigr),
$$
and continuity results for $\hR$ and $\hT,$ we  get
\begin{equation*}
\|\hcT_{X_{k,0}}\Omega_0-\div(X_{k,0}\Omega_0)\|_{\dot B^{\frac{3}{q}+\ep-3}_{q,1}}\lesssim \|\Omega_0\|_{\hB^{\frac{3}{q}-2}_{q,1}} \|X_{k,0}\|_{\hsC^\ep}, \,\,\, \hbox{for any}\,\, q \in \Big]\frac{3}{2}, \frac{3}{2-\ep} \Big[.
\end{equation*}
Thus, remembering \eqref{cdt:q0} and $0<\delta < \ep$, we have to choose some  $p$ satisfying \eqref{eq:p3},
 such that the following standard embedding  holds
\begin{equation}\label{cdt:q1}
\dot B^{\frac{3}{q}+\ep-3}_{q,1} \hookrightarrow \dot B^{\frac{3}{p}+\ep-3}_{p,1} \,\,
\hbox{for some}\,\,\, q \in  \Big]\frac{3}{2}, \frac{3}{2-\delta} \Big[ \,\,\,\hbox{with}\,\,\, q \leq p.
\end{equation}
\smallbreak
Now, because $\d_{X_{k,0}}\mathds{1}_{\mathcal{D}_0}\equiv0$ and
$\wt\Omega_0$ is in $B^{\delta_{\star}}_{\infty,1}$ for all  $0 < \delta_{\star}<\delta,$
Corollary \ref{cor:dXfg_3D}   yields, 
 $$\d_{X_{k,0}} \Omega_0=\div(X_{k,0}\otimes\Omega_0)=\div(X_{k,0}\otimes
 \wt\Omega_0\,\mathds{1}_{\mathcal{D}_0})
\in \dot B^{\delta_{\star}-1}_{q,1}\ \hbox{ for all }\ q\geq1.$$
 %Again, from the compact support property and Proposition \ref{prop:besov-holder}, we see that 
 %\begin{equation}\label{cdt:q2}
 %\d_{X_{k,0}} \Omega_0 \in \hB^{\delta_{\star}-1}_{q,1} \hookrightarrow  \hB^{\frac{3}{q}+\ep-3}_{q,1},  
%\hbox{for any}\,\,\,q \in \big]\frac{3}{2-\ep+\delta_{\star}}, \infty \big[.
 %\end{equation}
 One can thus conclude that $\d_{X_{k,0}}u_0\in\dot B^{\frac3p-2+\ep}_{p,1}$ for any index $p$ satisfying
$p\geq q$  with  $q$ satisfying  Condition \eqref{cdt:q1} and $\frac3q+\ep-2=\delta^*\in]0,\delta[.$ 

As one can require in addition  $p$ to fulfill  \eqref{eq:p3}, 
Proposition \ref{prop:NN}  applies with the family $(X_{k,0})_{1\leq k\leq3}.$
Denoting by  $(X_k)_{1\leq k\leq3}$ 
the corresponding family of  divergence free vector fields in $\cC^{0,\ep}$ given by  \eqref{eq:X} with initial data $X_{0,k},$ 
and introducing $Y_1:=X_3\wedge X_1,$ $Y_2:=X_3\wedge X_1$
 and $Y_3=X_1\wedge X_2,$  we discover that for $\alpha=1,2,3,$ 
 \begin{equation}\label{eq:Y}
\left\{\begin{array}{l}
\d_t Y_\alpha +u\cdot \nabla Y_\alpha = -  \nabla u\cdot Y_\alpha, \\[1ex]
 (Y_\alpha)|_{t=0}=\d_\alpha f_0\, \nabla f_0.
\end{array}
\right.
\end{equation}
 From \eqref{eq:Df}, it is clear that the time-dependent
 vector field  $\bigl(\d_\alpha f_0(\psi^{-1}_t) \big)\,\nabla f_t$ also satisfies \eqref{eq:Y}, 
 hence we have, by uniqueness, $Y_\alpha (t,\cdot) = \big((\d_\alpha f_0) (\psi^{-1}_t) \big)\nabla f_t.$
So finally,
 $$
 \big|\nabla f_0\circ \psi_t^{-1}\bigr|^2\,\nabla f_t=\sum_{\alpha=1}^3 Y_\alpha(t,\cdot)\,\d_\alpha f_0\circ\psi_t^{-1}.
 $$
  As $\psi_t^{-1}$ is $\cC^1$ and as both  $Y_\alpha$ and  $\nabla f_0$ 
  are in $\cC^{0,\ep},$ one can
  conclude that $ \nabla f_t$ is $\cC^{0,\ep}$ in some  neighborhood  of $\d\cD_0.$
  Therefore $\cD_t$ remains of class $\cC^{1,\ep}$ for all time. 
\end{proof}

\begin{rmk} 
 In the  3-D case, the mean free assumption on initial vorticity is not required, but 
 one cannot  consider constant vortex patterns as in the 2-D case.
 Let us also emphasize that, as for the Boussinesq system studied in \cite{DanZhx2016}, a similar statement may be proved
 in higher dimension. 
\end{rmk}

%%%%%%%%%%%%%%%%%%%%%%%%%%%%%%%%%%%%%%%%%%%%%%%%%%%%%%%%%%%%%%%%%%%%%%%%%%%%%%%%%%%  
 
\section{The proof  of persistence of striated regularity}\label{sect:3}

That section is devoted to the proof of Theorem  \ref{thm:main_dINS}.
The first step is to apply  Theorem \ref{thm:wellposed_INS}. From it, we get a unique global solution $(\rho,u,\nabla P)$ with 
$\rho\in \cC_b\big(\R_+;\cM(\dot B^{\frac Np-1}_{p,1})\big)$ and $(u,\nabla P)\in \dot E_p,$
satisfying \eqref{est:up}.
Because the product of functions maps $\dot B^{\frac Np-1}_{p,1}\times\dot B^{\frac Np}_{p,1}$
to $\dot B^{\frac Np-1}_{p,1},$   we deduce that the material derivative 
$D_tu=\d_tu+u\cdot\nabla u$ is also bounded by the right-hand side of \eqref{est:up}.
So finally,
\begin{equation}\label{es:u_INS}
\|(u, \nabla P)\|_{\hE_p} + \|D_t u\|_{L^1_t \big(\hB^{\frac{N}{p}-1}_{p,1}\big)} 
\leq C\|u_0\|_{\hB^{\frac{N}{p}-1}_{p,1} }.
\end{equation}
 
 In order to complete the proof of the theorem, it
is only a matter of showing that the additional multiplier and striated regularity properties are conserved
for all positive times.  In fact, we shall mainly concentrate on the proof of a priori estimates
for the  corresponding norms, just
explaining at the end of this section how a suitable regularization 
process allows to make it rigorous.

%%%%%%%

%One can easily verify that if $N \geq 2$ and $p \in ]\frac{N}{2}, \frac{N}{1-\ep}[,$ then 
%$$\frac{N}{p}-1, \frac{N}{p}+\ep -2  \in \,  ]-\min\{\frac{N}{p'},1\},\min\{ \frac{N}{p}, 1\}[.$$
%Recall our assumptions on $]\ep, p [$ in Theorem \ref{thm:main_dINS}, i.e. 
%\begin{equation}\label{cdt:p}
%\frac{N}{p}+\ep-2 \, \in \, ]\max\{ -1, \ep -\frac{3}{2}\}, \, \ep[, \quad \mbox{and} \quad \frac{N}{p}-1 \, \in \, ]\max\{ -\ep, -\frac{1}{2}\}, \, 1[,
%\end{equation}
%and also the notations in Lemma \ref{lemma:multiplier}
%\begin{equation}\label{es:dflow}
%C_{\psi_u^{\pm 1},s} \equiv C_{\psi_u^{\pm 1},s,p,1} \leq  \|D \psi_u^{\pm 1}\|^{|s|+N}_{L^{\infty}} \leq C.
%\end{equation} 

\subsection{Bounds involving multiplier norms}

As already pointed out in the introduction, because $\nabla u$ is in $L^1(\R_+;\dot B^{\frac Np}_{p,1})$
and $\dot B^{\frac Np}_{p,1}$ is embedded in $\cC_b,$ the flow $\psi$ of $u$ is $\cC^1$ and  we have 
for all $t\geq0,$ owing to \eqref{est:up},
\begin{equation}\label{es:dflow}
\|\nabla \psi_t^{\pm1}\|_{L^\infty}\leq \exp\biggl(\int_0^t\|\nabla u\|_{L^\infty}\,d\tau\biggr)
\leq C
\end{equation}
for  a suitably large universal constant $C.$
\medbreak
Now, from the mass conservation equation and \eqref{eq:dXrho},  we gather that
$$
\rho(t,\cdot)=\rho_0\circ\psi_t^{-1}\quad\hbox{and}\quad
(\d_X\rho)(t,\cdot)=(\d_{X_0}\rho_0)\circ\psi_t^{-1}. 
$$
Hence $\|\rho(t,\cdot)\|_{L^\infty}$ is time independent.
Furthermore,  Lemma \ref{lemma:multiplier} and Condition \eqref{cdt:p}  guarantee that for all  $t\in\R_+,$
 \begin{eqnarray}\label{es:1-rho_1}
\|\rho(t) -1 \|_{\cM \big(\hB^{\frac{N}{p}-1}_{p,1} \big)} &\!\!\!\!\leq\!\!\!\!& C \|\rho_0 -1 \|_{\cM \big(\hB^{\frac{N}{p}-1}_{p,1}\big)},\\\label{es:1-rho_2}
\|\rho(t) -1 \|_{\cM \big(\hB^{\frac{N}{p}+\ep-2}_{p,1}\big)} &\!\!\!\!\leq\!\!\!\!& C \|\rho_0 -1 \|_{\cM \big(\hB^{\frac{N}{p}+\ep-2}_{p,1}\big)},\\\label{es:dXrho}
\qquad\|(\d_X \rho)(t)\|_{\cM \big(\hB^{\frac{N}{p}-1}_{p,1} \to \hB^{\frac{N}{p}+\ep-2}_{p,1}\big)} &\!\!\!\!\leq\!\!\!\!& C \|\d_{X_0} \rho_0\|_{ \cM \big(\hB^{\frac{N}{p}-1}_{p,1} \rightarrow \hB^{\frac{N}{p}+\ep-2}_{p,1}\big) }.
\end{eqnarray}

%%%%%%%%%%%%%%%%%%%%

\subsection{ Estimates for  the striated regularity}
Recall that $\hcT_X u$ satisfies the Stokes-like system  \eqref{eq:dXINS}. 
As  $\hcT_X u$ need not be divergence free,  to enter into the standard theory, we  set 
 $$ v:= \hcT_X u - w\quad\hbox{with}\quad w:= \hT_{\d_k X} u^k-\hT_{\div X}u.$$
Denoting  
$\widetilde{g} := g- \rho u\cdot \nabla \hcT_X u -(\rho\d_t w - \Delta w)$ with  $g$ defined in \eqref{eq:g_dINS}, we see that  $v$ satisfies:
\begin{equation} \label{eq:v_dINS} \tag{$S$}
\left\{ \begin{array}{l}
\rho \d_t v - \Delta v + \nabla \hcT_X P = \widetilde{g},\\
\div v = 0,\\
v|_{t=0}=v_0.
\end{array}\right.
\end{equation}

We shall decompose the proof of  a priori estimates for striated regularity into three steps.
The first one is dedicated to bounding $\widetilde{g}$ (which mainly requires the commutator estimates
of the appendix). In the second step, we take advantage of the smoothing effect of the heat
flow so as to estimate $v.$ 
In the third step, we revert to $\hcT_X u$ and eventually  bound $X.$ 

%%%%%%%%%%%%%%%%%%%%%%%%%%%%%%%%%%%%%%%%%%%%%%%%%%%%%%%%%%%%%%%%%%%% 

\subsubsection*{First step: bounds of ${\wt g}$} Recall that $\widetilde{g} := g- \rho u\cdot \nabla \hcT_X u -(\rho\d_t w - \Delta w)$ with
\begin{equation*}
  g= - \rho [\hcT_{X}, D_t]u +[\hcT_X,\Delta] u - [\hcT_X, \nabla] P + (\d_X-\hcT_X) (\Delta u - \nabla P) -\d_X \rho D_t u + \rho (\hcT_X-\d_X)D_t u.
\end{equation*}

The first term of $g$ may be bounded according to Proposition \ref{prop:ce_dINS} and to the definition of multiplier spaces. We get, under assumption \eqref{cdt:p_X},
\begin{multline}\label{es:g1_dINS}
 \|\rho [\hcT_X, D_t ]u\|_{\hB^{\frac{N}{p}+\ep-2}_{p,1}}  \lesssim \|\rho\|_{\cM\big(\hB^{\frac{N}{p}+\ep-2}_{p,1}\big)}  \Big(  \|u\|_{\hsC^{-1}}\|\hcT_X u\|_{ \hB^{\frac{N}{p}+\ep}_{p,1}} \\
  + \|u\|_{\hB^{\frac{N}{p}+1}_{p,1}} \|\hcT_{X} u\|_{\hsC^{\ep-2}} 
 + \|u\|_{\hB^{\frac{N}{p}+1}_{p,1}} \|u\|_{\hB^{\frac{N}{p}-1}_{p,1}} \|X\|_{\hsC^{\ep}} \Big).
\end{multline}
 
 Next, thanks to the commutator estimates in Lemma \ref{lemma:ce1}, we have
\begin{equation}\label{es:g2_dINS}
 \|[\hcT_X,\Delta] u\|_{\hB^{\frac{N}{p}+\ep-2}_{p,1}} \lesssim   \|\nabla X\|_{\hsC^{\ep-1}}\|\nabla u\|_{\hB^{\frac{N}{p}}_{p,1}},
\end{equation}
\begin{equation}\label{es:g3_dINS}
 \|[\hcT_X, \nabla ]P\|_{\hB^{\frac{N}{p}+\ep-2}_{p,1}} \lesssim  \|\nabla X\|_{\hsC^{\ep-1}} \|\nabla P\|_{\hB^{\frac{N}{p}-1}_{p,1}}.
\end{equation}
Bounding the fourth term of $g$ stems from  \eqref{es:dXf_TXf}: we have
\begin{equation}\label{es:g4_dINS}
\| ( \hcT_X -\d_X ) (\Delta u -\nabla P ) \|_{\hB^{\frac{N}{p}+\ep-2}_{p,1}} \lesssim  \|(\Delta u ,\nabla P) \|_{\hB^{\frac{N}{p}-1}_{p,1}} \|X\|_{\hsC^{\ep}}.
\end{equation}
Then the definition of multiplier spaces yields
\begin{equation}\label{es:g5_dINS}
\|\d_X \rho D_t u\|_{\hB^{\frac{N}{p}+\ep-2}_{p,1}} \lesssim \|\d_X \rho\|_{\cM\big(\hB^{\frac{N}{p}-1}_{p,1} \rightarrow \hB^{\frac{N}{p}+\ep-2}_{p,1} \big)} \|D_t u\|_{\hB^{\frac{N}{p}-1}_{p,1}}.
\end{equation}
Finally, using again   \eqref{es:dXf_TXf}
 and the definition of multiplier spaces,  we may write
\begin{equation}\label{es:g6_dINS}
\| \rho (\hcT_X - \d_X) D_t u\|_{\hB^{\frac{N}{p}+\ep-2}_{p,1}} \lesssim \| \rho\|_{\cM\big( \hB^{\frac{N}{p}+\ep-2}_{p,1} \big)} \|X\|_{\hsC^\ep} \|D_t u\|_{\hB^{\frac{N}{p}-1}_{p,1}}.
\end{equation}
Putting together \eqref{es:g1_dINS} -- \eqref{es:g6_dINS} and integrating with respect to time, 
we end up with
\begin{multline}\label{es:g_dINS}
\|g\|_{L^1_t \big(\hB^{\frac{N}{p}+\ep-2}_{p,1} \big)} \lesssim \!
\int^t_0\! \|\rho\|_{\cM\big(\hB^{\frac{N}{p}+\ep-2}_{p,1}\big)}  \big( \|u\|_{\hsC^{-1}}\|\hcT_X u\|_{ \hB^{\frac{N}{p}+\ep}_{p,1}} \!+\! \|\nabla u\|_{\hB^{\frac{N}{p}}_{p,1}} \|\hcT_{X}u\|_{\hsC^{\ep-2}} \big)\,dt'          \\
 + \! \int_0^t  \|X\|_{\hsC^{\ep}} \Big(\! \big(\|\nabla u\|_{\hB^{\frac{N}{p}}_{p,1}} \|u\|_{\hB^{\frac{N}{p}-1}_{p,1}} 
 \!+\! \|D_t u\|_{\hB^{\frac{N}{p}-1}_{p,1}} \big)\|\rho\|_{\cM\big(\hB^{\frac{N}{p}+\ep-2}_{p,1}\big)} \!+\! \| ( \nabla^2 u, \nabla P)\|_{\hB^{\frac{N}{p}-1}_{p,1}} \Big)dt'   \\
 + \int_0^t\|\d_X \rho\|_{\cM\big(\hB^{\frac{N}{p}-1}_{p,1} \rightarrow \hB^{\frac{N}{p}+\ep-2}_{p,1} \big)} \| D_t u \|_{\hB^{\frac{N}{p}-1}_{p,1}} \,dt'.
\end{multline}

Bounding the second term  of $\widetilde{g}$ is obvious : taking advantage of  Bony's decomposition
\eqref{eq:bony}
and remembering that $\frac{N}{p}+\ep > 1 $ and that  $\div u =0$, we get
\begin{multline}\label{es:tg1_dINS}
\|\rho u\cdot \nabla \hcT_X u\|_{L^1_t \big(\hB^{\frac{N}{p}+\ep-2}_{p,1} \big)}  \lesssim  \int^t_0 \|\rho\|_{\cM\big(\hB^{\frac{N}{p}+\ep-2}_{p,1}\big)}    \big(   \|u\|_{\hsC^{-1}}\|\hcT_{X}u\|_{\hB^{\frac{N}{p}+\ep}_{p,1}}     \\
+\|u\|_{\hB^{\frac{N}{p}+1}_{p,1}}\|\hcT_{X}u\|_{\hsC^{\ep-2}} \big) \,dt'.
\end{multline} 
To bound the  last term  of $\wt g,$  we use the decomposition
$$
\rho \d_t w-\Delta w=  \rho (W_1+W_2)+ W_3,
$$
with 
$$
W_1:=  \hT_{\d_k X}\d_tu^k-\hT_{\div X}\d_tu ,\quad
W_2:= \hT_{\d_k \d_tX}u^k-\hT_{\div\d_tX}u , \quad
W_3:=\Delta\bigl(\hT_{\div X}u-\hT_{\d_k X}u^k\bigr).
$$
%It suffices to bound the components of $W_\alpha$ in $L^1_t\big(\hB^{\frac{N}{p}+\ep-2}_{p,1}\big)$  with $\alpha=1,2,3.$ 
%\medbreak

Continuity results for the paraproduct and the definition of
$\cM\big(\hB^{\frac{N}{p}+\ep-2}_{p,1}\big)$  ensure that
\begin{eqnarray}\label{eq:W1}
\|\rho W_1 \|_{L^1_t\big(\hB^{\frac{N}{p}+\ep-2}_{p,1}\big)} & \lesssim& \int^t_0 \|\rho\|_{\cM\big(\hB^{\frac{N}{p}+\ep-2}_{p,1}\big)} \|\nabla  X \|_{\hsC^{\ep-1}} \|\d_t u\|_{\hB^{\frac{N}{p}-1}_{p,1}}\,dt',\\\label{eq:W2}
\| \rho W_2 \|_{L^1_t\big(\hB^{\frac{N}{p}+\ep-2}_{p,1}\big)}  &\lesssim& \int^t_0 \|\rho\|_{\cM\big(\hB^{\frac{N}{p}+\ep-2}_{p,1}\big)} \|\d_t X\|_{\hsC^{\ep-2}}\|u\|_{\hB^{{\frac Np}+1}_{p,1}}\,dt',\\\label{eq:W3}
\| W_3\|_{L^1_t\big(\hB^{\frac{N}{p}+\ep-2}_{p,1}\big)} &\lesssim&\int^t_0  \| \nabla X \|_{\hsC^{\ep-1}}\|u\|_{\hB^{\frac{N}{p}+1}_{p,1}}\,dt'.
\end{eqnarray}
To estimate $\d_tX$ in \eqref{eq:W2},  we use the fact that
$$
\d_tX=-u\cdot\nabla X+\d_Xu=-\div(u\otimes X)+\d_Xu.
$$
Hence  using \eqref{eq:bony}, and
continuity results for the remainder and paraproduct operators, we get under Condition \eqref{cdt:p_X},
$$\|\d_tX \|_{\hsC^{\ep-2}}\lesssim 
 \|u\|_{\hB^{\frac{N}{p}-1}_{p,1}} \|X\|_{\hsC^{\ep}}+ \|\d_Xu\|_{\hsC^{\ep-2}}.$$
Therefore,  taking advantage of  \eqref{es:dXf_TXf} yields
\begin{equation}\label{eq:W4}
\| \rho W_2 \|_{L^1_t\big(\hB^{\frac{N}{p}+\ep-2}_{p,1}\big)}   \lesssim \int^t_0  \|\rho\|_{\cM\big(\hB^{\frac{N}{p}+\ep-2}_{p,1}\big)} (\|X\|_{\hsC^{\ep}}\|u\|_{\hB^{\frac{N}{p}-1}_{p,1}} + \|\hcT_X u\|_{\hsC^{\ep-2}}) \|\nabla u\|_{\hB^{\frac{N}{p}}_{p,1}}\, dt'.
\end{equation}
Combining  \eqref{eq:W1}, \eqref{eq:W2} and \eqref{eq:W4}, we eventually obtain 
\begin{multline}\label{es:tg2_dINS}
\|\rho \d_t w - \Delta w \|_{L^1_t\big(\hB^{\frac{N}{p}+\ep-2}_{p,1}\big)} \lesssim \int^t_0 \|\hcT_X u\|_{\hsC^{\ep-2}}  \|\nabla u\|_{\hB^{\frac{N}{p}}_{p,1}}  \|\rho\|_{\cM\big(\hB^{\frac{N}{p}+\ep-2}_{p,1}\big)} \,dt'\\
+\int^t_0 \|X\|_{\hsC^{\ep}} \Big(  \big( \|\rho\|_{\cM\big(\hB^{\frac{N}{p}+\ep-2}_{p,1}\big)} \|u\|_{\hB^{\frac{N}{p}-1}_{p,1}} +1 \big) \|\nabla u\|_{\hB^{\frac{N}{p}}_{p,1}} +  \|\rho\|_{\cM\big(\hB^{\frac{N}{p}+\ep-2}_{p,1}\big)} \|\d_t u\|_{\hB^{\frac{N}{p}-1}_{p,1}} \Big) \, dt'. 
\end{multline}
Putting together estimate \eqref{es:g_dINS}, \eqref{es:tg1_dINS} and \eqref{es:tg2_dINS}, we eventually obtain
\begin{multline}\label{es:tg_dINS}
\|\widetilde{g}\|_{L^1_t\big(\hB^{\frac{N}{p}+\ep-2}_{p,1}\big)} \lesssim  
\!\int^t_0 \!\|\rho\|_{\cM\big(\hB^{\frac{N}{p}+\ep-2}_{p,1}\big)}  \big( \|u\|_{\hsC^{-1}}\|\hcT_X u\|_{ \hB^{\frac{N}{p}+\ep}_{p,1}}\! +\! \|\nabla u\|_{\hB^{\frac{N}{p}}_{p,1}} \|\hcT_{X}u\|_{\hsC^{\ep-2}} \big)\,dt'          \\
 +  \int_0^t  \|X\|_{\hsC^{\ep}} \big( \|\nabla u\|_{\hB^{\frac{N}{p}}_{p,1}} \|u\|_{\hB^{\frac{N}{p}-1}_{p,1}} + \|(\d_t u,D_t u)\|_{\hB^{\frac{N}{p}-1}_{p,1}} \big)\|\rho\|_{\cM\big(\hB^{\frac{N}{p}+\ep-2}_{p,1}\big)} \,dt'               \\
 +  \int_0^t  \|X\|_{\hsC^{\ep}} \|(\nabla^2 u,  \nabla P)\|_{\hB^{\frac{N}{p}-1}_{p,1}}  \,dt'               + \int_0^t\|\d_X \rho\|_{\cM\big(\hB^{\frac{N}{p}-1}_{p,1} \rightarrow \hB^{\frac{N}{p}+\ep-2}_{p,1} \big)} \| D_t u \|_{\hB^{\frac{N}{p}-1}_{p,1}} \,dt'.\\ 
  \end{multline}
  
%%%%%%%%%%%%%%%%%%%%%%%%%%%%%%%%%%%%%%%%%%%%%%%%%%%%%%%%%%%%%%%%%%%%%%%%%%%%%%%%%%%

\subsubsection*{Second step: bounds of ${v}$}  
 
 We now want to bound $v$ in  $\wt{L}^{\infty}_t \big(\hB^{\frac{N}{p}+\ep-2}_{p,1} \big) \cap L^{1}_t \big( \hB^{\frac{N}{p}+\ep}_{p,1} \big),$  knowing \eqref{es:tg_dINS}.
 This will  follow from the smoothing properties of the heat flow. 
 More precisely,  introduce the projector $\mathbb P$ over divergence-free vector fields,
  and apply  $\mathbb{P}\hDelta_j$  (with $j\in\Z$) to the equation \eqref{eq:v_dINS}. We get
\begin{equation*}
\left\{ \begin{array}{l}
\d_t \hDelta_jv - \Delta \hDelta_jv = \mathbb{P}\hDelta_j(\widetilde{g} + (1-\rho)\d_t v)\\[1ex]
\hDelta_jv|_{t=0} = \hDelta_jv_{0}.
\end{array}\right.
\end{equation*}
 Lemma 2.1 in  \cite{Che1999} implies that if $p\in [1, \infty],$
\begin{align*}
\|\hDelta_jv(t)\|_{L^{p}} \leq e^{-ct 2^{2j}} \|\hDelta_jv_{0}\|_{L^{p}} + C \int^{t}_{0} e^{-c(t-t') 2^{2j}} 
\|\hDelta_j(\widetilde{g} + (1-\rho)\d_t v)(t')\|_{L^{p}}\,dt'.
\end{align*}
Therefore,  taking the supremum over $j\in \Z,$  using the fact that
\begin{equation*}
\d_t v = \Delta v  + \mathbb{P} \big(\widetilde{g} + (1-\rho)\d_t v\big)
\end{equation*}
and that $\mathbb{P}:\hB^{\frac{N}{p}+\ep-2}_{p,1}\to\hB^{\frac{N}{p}+\ep-2}_{p,1},$
we find that 
\begin{multline}\label{es:v1_dINS}
\|v\|_{\wt{L}^{\infty}_t \big(\hB^{\frac{N}{p}+\ep-2}_{p,1} \big)} 
+ \|v\|_{L^{1}_t \big(\hB^{\frac{N}{p}+\ep}_{p,1}\bigr)} + 
\|\d_tv\|_{L^{1}_t \big(\hB^{\frac{N}{p}+\ep-2}_{p,1}\big)} \\
\lesssim \|v_0\|_{\hB^{\frac{N}{p}+\ep-2}_{p,1} } + \|\wt{g}\|_{L^{1}_t \big(\hB^{\frac{N}{p}+\ep-2}_{p,1} \big)}+ \|(1-\rho)\d_t v\|_{L^{1}_t \big(\hB^{\frac{N}{p}+\ep-2}_{p,1} \big)}.
\end{multline}
The smallness condition \eqref{cdt:u0_small} combined with Inequality \eqref{es:1-rho_2} ensure 
that the last term of \eqref{es:v1_dINS} may  be absorbed by the left-hand side, and we thus 
end up with 
\begin{equation*}
\|v\|_{\wt{L}^{\infty}_t \big(\hB^{\frac{N}{p}+\ep-2}_{p,1} \big) \cap L^{1}_t \big(\hB^{\frac{N}{p}+\ep}_{p,1} \big)} + \|\d_t v\|_{L^{1}_t \big(\hB^{\frac{N}{p}+\ep-2}_{p,1} \big)} \lesssim \|v_0\|_{\hB^{\frac{N}{p}+\ep-2}_{p,1} } + \|\wt{g}\|_{L^{1}_t \big(\hB^{\frac{N}{p}+\ep-2}_{p,1} \big)}.
\end{equation*}
Next, we use the fact that by  definition of $v_0,$
\begin{eqnarray*}
v_0 &=& \hcT_{X_0} u_0 - \hT_{\d_k X_0} u_0^k + \hT_{\div X_0}u_0 \\
&=& \d_{X_0} u_0- \hT_{\d_k u_0}X^k_0 - \d_k \hR(X^k_0, u_0) + \hR(\div X_0, u_0) - \hT_{\d_k X_0} u_0^k + \hT_{\div X_0}u_0.
\end{eqnarray*}
Hence continuity results for the paraproduct yield, under Condition \eqref{cdt:p_X},
\begin{equation*}
\|v_0\|_{\hB^{\frac{N}{p}+\ep-2}_{p,1} } \lesssim \|\d_{X_0} u_0\|_{\hB^{\frac{N}{p}+\ep-2}_{p,1} } +\| X_0\|_{\hsC^{\ep}} \|u_0\|_{\hB^{\frac{N}{p} -1}_{p,1}}.
\end{equation*}
Thus 
\begin{multline} \label{es:v_dINS}
\|v\|_{\wt{L}^{\infty}_t \big(\hB^{\frac{N}{p}+\ep-2}_{p,1} \big) \cap L^{1}_t \big(\hB^{\frac{N}{p}+\ep}_{p,1} \big)} + \|\d_t v\|_{L^{1}_t \big(\hB^{\frac{N}{p}+\ep-2}_{p,1} \big)}  \lesssim \|\d_{X_0} u_0\|_{\hB^{\frac{N}{p}+\ep-2}_{p,1} } \\+\| X_0\|_{\hsC^{\ep}} \|u_0\|_{\hB^{\frac{N}{p} -1}_{p,1}} 
+ \|\wt{g}\|_{L^{1}_t \big(\hB^{\frac{N}{p}+\ep-2}_{p,1} \big)}.
\end{multline}
%%%%%%%%%%%%%%%%%%%%%%%%%%%%%%%%%%%%%%%%%%%%%%%%%%%%%%%%%%%%%%%%%%%%%%%%%%%%%%%%%%%%%5

\subsubsection*{Third step: bounds for striated regularity}
 Remembering that
$$\hcT_X u=v+w\quad\hbox{with}\quad w=\hT_{\d_k X}u^k - \hT_{\div X}u,$$
it is now easy to bound the following quantity:
\begin{equation*}
\sH(t) := \|\hcT_X u\|_{\wt{L}^{\infty}_t \big(\hB^{\frac{N}{p}+\ep-2}_{p,1} \big)}+\|\hcT_X u\|_{L^{1}_t \big( \hB^{\frac{N}{p}+\ep}_{p,1} \big)}+\|\nabla \hcT_X P\|_{L^{1}_t \big( \hB^{\frac{N}{p}+\ep-2}_{p,1} \big)}.
\end{equation*}
Indeed,  we have
\begin{equation}\label{eq:dPi}
\nabla \hcT_X P= ({\rm Id}- \mathbb{P}) (\widetilde{g} -\rho \d_t v), 
\end{equation} 
and thus $\|\nabla \hcT_X P\|_{L^{1}_t \big( \hB^{\frac{N}{p}+\ep-2}_{p,1} \big)}$
may be bounded by the right-hand side of \eqref{es:v_dINS}. Note also that 
 continuity results for paraproduct operators guarantee that
$$
\begin{aligned}
\|w\|_{\wt{L}^{\infty}_t \big(\hB^{\frac{N}{p}+\ep-2}_{p,1} \big)} &\lesssim
 \|u\|_{ \wt{L}^{\infty}_t (\hB^{\frac{N}{p}-1}_{p,1})}\|X\|_{L^{\infty}_t (\hsC^{\ep})},\\ 
 \|w\|_{L^{1}_t \big(\hB^{\frac{N}{p}+\ep}_{p,1} \big)} &\lesssim 
  \int^t_0 \|u\|_{\hB^{\frac{N}{p}+1}_{p,1}} \|\nabla X\|_{\hsC^{\ep-1}}\,dt'.
\end{aligned}
$$
Hence we have
\begin{multline}\label{es:w_dINS1}
\sH(t)\lesssim  \|\d_{X_0} u_0\|_{\hB^{\frac{N}{p}+\ep-2}_{p,1} } 
+\| X_0\|_{\hsC^{\ep}} \|u_0\|_{\hB^{\frac{N}{p} -1}_{p,1}} 
+ \|\wt{g}\|_{L^{1}_t \big(\hB^{\frac{N}{p}+\ep-2}_{p,1} \big)}\\
+ \|u\|_{ \wt{L}^{\infty}_t (\hB^{\frac{N}{p}-1}_{p,1})\cap L^1_t(\hB^{\frac Np+1}_{p,1})}
\|X\|_{L_t^\infty(\hsC^\ep)}.
\end{multline}
Because  $X$ satisfies \eqref{eq:X},  standard H\"older estimates for transport equations imply that 
$$
\| X \|_{L^{\infty}_t (\hsC^{\ep})}  \leq \|X_0\|_{\hsC^{\ep}} + \int^t_0 \|\nabla u\|_{L^\infty} \| X \|_{\hsC^{\ep}}\, dt' + \int_0^t\|\d_X u\|_{\hsC^{\ep}}\,dt'.
$$
Now, recall that
$$
\d_Xu-\dot\cT_Xu=\dot T_{\d_ku}X^k+\dot R(\d_ku,X^k)
$$
whence, using standard continuity results for operators $\dot T$ and $\dot R,$ and embedding,
\begin{equation}\label{eq:paraX}
\|\dot\cT_Xu-\d_X u\|_{\hsC^{\ep}}\lesssim 
\|\dot\cT_Xu-\d_X u\|_{\hB^{\frac{N}{p}+\ep}_{p,1}}\lesssim 
\|\nabla u\|_{\dot B^{\frac Np}_{p,1}}\|X\|_{\hsC^\ep}.
\end{equation}
Therefore we have
\begin{equation}\label{es:X_dINS}
\| X \|_{L^{\infty}_t (\hsC^{\ep})}
 \leq \|X_0\|_{\hsC^{\ep}} + \int^t_0 \|\nabla u\|_{\hB^{\frac{N}{p}}_{p,1}}  \| X \|_{\hsC^{\ep}}\, dt' + \|\hcT_X u\|_{L^1_t \big(\hB^{\frac{N}{p}+\ep}_{p,1}\big)}.
\end{equation}
Then,  using \eqref{es:u_INS} and plugging the above inequality in \eqref{es:w_dINS1}, we get
$$\displaylines{
\sH(t)\lesssim  \|\d_{X_0} u_0\|_{\hB^{\frac{N}{p}+\ep-2}_{p,1} } 
+\| X_0\|_{\hsC^{\ep}} \|u_0\|_{\hB^{\frac{N}{p} -1}_{p,1}} 
+ \|\wt{g}\|_{L^{1}_t \big(\hB^{\frac{N}{p}+\ep-2}_{p,1})} \hfill\cr\hfill
+\|u_0\|_{\dot B^{\frac Np-1}_{p,1}}\biggl(\|\hcT_X u\|_{L^1_t \big(\hB^{\frac{N}{p}+\ep}_{p,1}\big)}
+ \int^t_0 \|\nabla u\|_{\hB^{\frac{N}{p}}_{p,1}}  \| X \|_{\hsC^{\ep}}\, dt'\biggr)\cdotp} 
$$
Choosing $c$ small enough in \eqref{cdt:u0_small}, we see that the first term of 
the second line may be absorbed by the left-hand side. 
Therefore, setting  $$\sK(t) := \sH(t)+\| X \|_{L^{\infty}_t (\hsC^{\ep})}$$
and using again \eqref{es:X_dINS} and the smallness of $u_0,$
$$
\sK(t)\lesssim  \|\d_{X_0} u_0\|_{\hB^{\frac{N}{p}+\ep-2}_{p,1} } 
+\| X_0\|_{\hsC^{\ep}} + \|\wt{g}\|_{L^{1}_t \big(\hB^{\frac{N}{p}+\ep-2}_{p,1})}
+ \int^t_0 \|\nabla u\|_{\hB^{\frac{N}{p}}_{p,1}}  \| X \|_{\hsC^{\ep}}\, dt'.
$$
In order to close the estimates,  it suffices to bound $\wt g$ by means of \eqref{es:tg_dINS}. 
Then the above inequality becomes, after using \eqref{es:1-rho_2} and \eqref{es:dXrho},
$$
\displaylines{\sK(t)\lesssim  \|\d_{X_0} u_0\|_{\hB^{\frac{N}{p}+\ep-2}_{p,1} } 
+\| X_0\|_{\hsC^{\ep}}\hfill\cr\hfill +
\int^t_0 \!\|\rho\|_{\cM\big(\hB^{\frac{N}{p}+\ep-2}_{p,1}\big)}  \big( \|u\|_{\hsC^{-1}}\|\hcT_X u\|_{ \hB^{\frac{N}{p}+\ep}_{p,1}}\! +\! \|\nabla u\|_{\hB^{\frac{N}{p}}_{p,1}} \|\hcT_{X}u\|_{\hsC^{\ep-2}} \big)\,dt'          \cr
 +  \int_0^t  \|X\|_{\hsC^{\ep}} \big( \|\nabla u\|_{\hB^{\frac{N}{p}}_{p,1}} \|u\|_{\hB^{\frac{N}{p}-1}_{p,1}} + \|(\d_t u,D_t u)\|_{\hB^{\frac{N}{p}-1}_{p,1}} \big)\|\rho_0\|_{\cM\big(\hB^{\frac{N}{p}+\ep-2}_{p,1}\big)} \,dt'               \cr
 +  \int_0^t  \|X\|_{\hsC^{\ep}} \|(\nabla^2 u,  \nabla P)\|_{\hB^{\frac{N}{p}-1}_{p,1}}  \,dt' 
          + \int_0^t\|\d_{X_0}\rho_0\|_{\cM\big(\hB^{\frac{N}{p}-1}_{p,1} \rightarrow \hB^{\frac{N}{p}+\ep-2}_{p,1} \big)} \| D_t u \|_{\hB^{\frac{N}{p}-1}_{p,1}} \,dt'.}
  $$
  The smallness of $\rho_0$ and $u_0$  implies that the second line may be 
  absorbed by the l.h.s. Therefore using the bounds for $\d_t u$
  and $D_tu$ in \eqref{es:u_INS},  we eventually get
  $$
\displaylines{\sK(t)\lesssim  \|\d_{X_0} u_0\|_{\hB^{\frac{N}{p}+\ep-2}_{p,1} } 
+\| X_0\|_{\hsC^{\ep}} +
\|\rho_0\|_{\cM\big(\hB^{\frac{N}{p}+\ep-2}_{p,1}\big)}\|u_0\|_{\dot B^{\frac Np-1}_{p,1}}
\biggl(1+\int_0^t \|\nabla u\|_{\hB^{\frac{N}{p}}_{p,1}}\|X\|_{\hsC^{\ep}}\,d\tau\biggr)\hfill\cr\hfill
+ \int_0^t  \|X\|_{\hsC^{\ep}} \|(\nabla^2 u,  \nabla P)\|_{\hB^{\frac{N}{p}-1}_{p,1}}  \,dt' + 
\|\d_{X_0} \rho_0\|_{\cM\big(\hB^{\frac{N}{p}-1}_{p,1} \rightarrow \hB^{\frac{N}{p}+\ep-2}_{p,1} \big)}
\int_0^t \|D_tu \|_{\hB^{\frac{N}{p}-1}_{p,1}}\,d\tau.}
$$
It is now easy to conclude by means of Gronwall lemma and \eqref{es:u_INS}. Using
once again the smallness of $u_0,$ we get
\begin{multline}\label{eq:sK}
\sK(t)\lesssim  \|\d_{X_0} u_0\|_{\hB^{\frac{N}{p}+\ep-2}_{p,1} } 
+\| X_0\|_{\hsC^{\ep}} \\+
\bigl(\|\d_{X_0} \rho_0\|_{\cM\big(\hB^{\frac{N}{p}-1}_{p,1} \rightarrow \hB^{\frac{N}{p}+\ep-2}_{p,1} \big)}
+\|\rho_0\|_{\cM\big(\hB^{\frac{N}{p}+\ep-2}_{p,1}\big)}\bigr)\|u_0\|_{\dot B^{\frac Np-1}_{p,1}}.
\end{multline}

{}From \eqref{eq:paraX}, we gather  that  $\d_Xu$ is bounded by the right-hand side of \eqref{eq:sK}. 
Next, in order to control the whole nonhomogeneous H\"older norm of $X,$ it suffices to remember that 
$$
\|X\|_{\cC^{0,\ep}}=\|X\|_{L^\infty}+\|X\|_{\dot\sC^\ep}
$$
and that Relation \eqref{eq:Xt} together with \eqref{es:dflow} directly yield
$$
\|X_t\|_{L^\infty}\leq\|\d_{X_0}\psi_t\|_{L^\infty}\leq C\|X_0\|_{L^\infty}.
$$
Finally,   to estimate $\d_X\nabla P,$ we use Inequality \eqref{es:dXf_TXf} and get
$$
\|\d_X\nabla P-\nabla\hcT_XP\|_{L_t^1\big(\hB^{\frac{N}{p}+\ep-2}_{p,1}\big)}\lesssim \|X\|_{L_t^\infty(\hsC^\ep)}\|\nabla P\|_{L^1_t(\hB^{\frac{N}{p}-1}_{p,1})}.$$
Therefore $\|\d_X\nabla P\|_{L_t^1\big(\hB^{\frac{N}{p}+\ep-2}_{p,1}\big)}$  
may be bounded like $\sK(t).$

%%%%%%%%%%%%

\subsection{The regularization process}

In all the above computations, we  implicitly assumed that 
$X$ and $\d_Xu$ were in $L^\infty_{loc}(\R_+;\cC^{0,\ep})$ and 
$L^1_{loc}(\R_+;\cC^{0,\ep}),$ respectively.
However, Theorem \ref{thm:wellposed_INS} just ensures continuity of those vector-fields, not
H\"older regularity.  

 To overcome that difficulty, one may smooth out the initial velocity (not the density, not to
 destroy the multiplier hypotheses) by setting for example $u_0^n:=\dot S_nu_0.$ 
       Then Condition \eqref{cdt:u0_small} is satisfied by $(\rho_0,u_0^n)$ 
       and, as in addition        
       $u_0^n$ belongs to all Besov spaces  $\dot B^{\frac Np-1}_{\wt p,r}$ with $\wt p\geq p$ and $r\geq1,$
       one can 
       apply\footnote{That paper is dedicated to the half-space, but having the same result
       in the whole space setting is much easier.}  
        \cite[Th. 1.1]{DZ} for solving (INS)  with initial data $(\rho_0,u_0^n).$
    This provides us with a unique        
             global solution $(\rho^n,u^n,\nabla P^n)$ which, among others, satisfies
       $$\nabla u^n\in L^r(\R_+;\dot B^{\frac Np}_{\wt p,r})
       \quad\hbox{for all }\    r\in]1,\infty[\ \hbox{ and }\ 
       \max\biggl(p,\frac{Nr}{3r-2}\biggr)\leq \wt p \leq \frac {Nr}{r-1}\cdotp$$
       By taking $r$ sufficiently close to $1$ and using embedding, we see that 
       this implies that $\nabla u^n$ is in $L^1_{loc}(\R_+;\hsC^{0,\delta})$
       for all $0<\delta<1$ and thus    
       the corresponding flow $\psi^n$  is (in particular)  in $\cC^{1,\ep}.$ This  ensures, thanks to 
 \eqref{eq:Xt}, that $X^n$ is in $L^\infty_{loc}(\R_+;\cC^{0,\ep})$ and thus that
 $\d_{X^n}u^n$ is in $L^1_{loc}(\R_+;\cC^{0,\ep}).$
 
 {}From the previous steps and the fact that the data $(\rho_0,u_0^n)$ satisfy \eqref{cdt:u0_small} uniformly,
 we get  uniform  bounds for $\rho^n,$ $u^n,$ $\nabla P^n$ and $X^n,$ and standard arguments
 thus allow to show that $u^n$ tends to $u$ in $L^1_{loc}(\R_+;L^\infty)$ and
 thus $(\psi^n-\psi)\to0$ in $L^\infty_{loc}(\R_+;L^\infty).$ 
 Interpolating with the uniform bounds and using standard functional analysis 
 arguments, one can eventually conclude that 
  $X^n\to X$ in $L^\infty_{loc}(\R_+;\cC^{0,\ep'})$ for all $\ep'<\ep$
  (and similar results for $(u^n)_{n\in\N}$) and that  all the estimates of the previous steps are satisfied. 
  The details are left to the reader.\qed

%%%%%%%%%%%%%%%%%%%%%%%%%%%%%%%%%%%%%%%%%%%%%%%%%%%%%%%
 
\begin{appendix}

\section{Multiplier spaces}

The following  relationship between the nonhomogeneous Besov spaces $B^{s}_{p,r}(\R^N)$  
and  the homogeneous Besov spaces $\hB^{s}_{p,r}(\R^N)$ for 
\emph{compactly supported} functions or distributions  has been established in \cite[Section 2.1]{DanM2015}.
\begin{prop}\label{prop:equiv_besov}
Let $(p,r) \in [1,\infty]^2$ and $s > -\frac{N}{p'}:= -N(1-\frac{1}{p})$ (or just $s\geq-\frac N{p'}$ if $r=\infty$).
 For any $u$ in the set  $\cE'(\R^N)$  of compactly supported distributions on $\R^N,$  we have
\begin{equation*}
u \in  B^{s}_{p,r}(\R^N) \Longleftrightarrow u\in \hB^{s}_{p,r}(\R^N).
\end{equation*}
Moreover, there exists a constant $C=C(s,p,r,N, {\rm Supp}\, u)$ such that
\begin{equation*}
C^{-1}\|u\|_{\hB^{s}_{p,r}} \leq \|u\|_{B^{s}_{p,r}} \leq C \|u\|_{\hB^{s}_{p,r}}.
\end{equation*}
\end{prop}
A simple consequence of Proposition \ref{prop:equiv_besov} and of standard embeddings for nonhomogeneous Besov spaces is that for any  $(s,p,r)$ as above, we have  
\begin{equation}\label{es:nonhom_hom}
\cE'(\R^N) \cap  \hB^{s+\delta}_{p,r}(\R^N) \hookrightarrow  \cE'(\R^N) \cap  \hB^{s}_{p,r}(\R^N) \quad \mbox{for any } ~\delta>0.
\end{equation}
We also used the following statement:
\begin{prop}\label{prop:besov-holder}
 Let $(p,s)$ be arbitrary in $[1,\infty]\times\R.$ 
 Then for all $u\in B^s_{\infty,1}(\R^N)\cap\cE'(\R^N),$ 
we have $u\in B^{s}_{p,1}(\R^N)$  and  there exists  $C=C(s,p, {\rm Supp}\,u)$ such that 
$$
\|u\|_{B^s_{p,1}}\leq C\|u\|_{B^s_{\infty,1}}.
$$
\end{prop}
\begin{proof}
Let $u$ be in $B^s_{\infty,1}(\R^N)$ with compact support, and fix some 
smooth cut-off function $\phi$ so that $\phi\equiv1$ on ${\rm Supp}\, u.$ 
Of course, being compactly and smooth, $\phi$ belongs to any nonhomogeneous Besov space.
Then, using decomposition \eqref{eq:bony} and the fact that $u=\phi u$, one can write
$$
u=T_\phi u+ T_u\phi+R(u,\phi).
$$
Because $\phi$ is in $L^p$ and $u,$ in $B^s_{\infty,1},$ standard continuity results for the paraproduct
ensure that $T_\phi u$ is in $B^s_{p,1}.$ For the second term, we just use that 
$u$ is in, say, $B^{\min(0,s)}_{\infty,1}$ and $\phi,$ in $B^{-\min(0,s)+s}_{p,1}$ hence 
$T_u\phi$ is in $B^s_{p,1}.$ For the remainder term, we use for instance the fact that $\phi$ is in $B^{|s|+\frac12}_{p,1}.$
Putting all those informations together completes the proof.
\end{proof}

The following result  was the key to bounding the density terms in our study of $(INS).$
\begin{lemma}\label{lemma:multiplier}
Let $(s,s_k,p,p_k,r, r_k) \in \,]-1,1[^2 \times [1,\infty]^4$ with $k=1,2,$  and $Z:\R^N \rightarrow \R^N$ be a $\cC^1$ measure preserving  diffeomorphism such that  $DZ$ and $DZ^{-1}$ are bounded. When we consider the homogeneous Besov space $\hB^s_{p,r}(\R^N)$ or $\hB^{s_k}_{p_k,r_k}(\R^N),$ we assume in addition that $s\in]-\frac N{p'},\frac Np[$ and $s_k\in]-\frac N{p'_k},\frac N{p_k}[$ for 
$k=1,2.$ Then we have: 
\begin{enumerate}
\item If $b^{s}_{p,r}(\R^N)$ stands for $B^{s}_{p,r}(\R^N)$ or $\hB^{s}_{p,r}(\R^N),$ then the mapping $u \mapsto u\circ Z $ is continuous on $b^{s}_{p,r}(\R^N)$: 
  there is a positive constant $C_{Z,s,p,r}$ such that 
\begin{equation}\label{es:composition_besov}
\|u\circ Z \|_{b^{s}_{p,r}} \leq  C_{Z,s,p,r} \|u \|_{b^{s}_{p,r}}.
\end{equation}

\item  If $b^{s_k}_{p_k, r_k}$ with $k=1,2,$ denote the same type of Besov spaces, then the mapping $\varphi \mapsto \varphi \circ Z $ is continuous on $\cM\big(b^{s_1}_{p_1,r_1}(\R^N) \rightarrow b^{s_2}_{p_2,r_2}(\R^N) \big),$ that is 
\begin{equation*}
\|\varphi \circ Z\|_{\cM\big(b^{s_1}_{p_1,r_1}\to b^{s_2}_{p_2,r_2} \big)} \leq C_{Z^{-1},1} C_{Z,2}  \|\varphi\|_{\cM\big(b^{s_1}_{p_1,r_1} \to b^{s_2}_{p_2,r_2} \big)}.
\end{equation*}
\item We have the following equivalence for any $\varphi \in \cE'(\R^N),$   
\begin{equation*}
\varphi \in \cM\big(B^{s_1}_{p_1,r_1}(\R^N) \rightarrow B^{s_2}_{p_2,r_2}(\R^N) \big) 
\Longleftrightarrow   
\varphi \in  \cM\big(b^{s_1}_{p_1,r_1}(\R^N) \rightarrow b^{s_2}_{p_2,r_2}(\R^N) \big) . 
\end{equation*}
Here $b^{s_1}_{p_1,r_1}$ and $b^{s_2}_{p_2,r_2}$ can be different type of Besov spaces but obey our convention on the index $s_k$ for homogeneous Besov space. 
\end{enumerate}
\end{lemma}

\begin{proof}
Item    (i) in the  case $b= \hB$ has been proved in   \cite[Lemma 2.1.1]{DanM2015}. 
 One may easily modify the proof  to handle  nonhomogeneous Besov spaces: 
use the finite difference characterization of \cite[Page 98]{Tri1992} if $s>0,$ argue 
by duality if $s<0$ and interpolate for the case $s=0$.  We get
  $C_{Z,s,p,r}\approx 1+ \|DZ\|^{s+\frac{N}{r}}_{L^{\infty}}$ if $s>0,$ and   $C_{Z,s,p,r}\approx 1+ \|DZ^{-1}\|^{-s+\frac{N}{r'}}_{L^{\infty}}$ if  $s<0.$  
\medbreak
Part (ii) is  immediate according to \eqref{eq:def_norm_multiplier} and \eqref{es:composition_besov}. Indeed we may write:
$$
\begin{aligned}
\|\varphi \circ Z\|_{\cM\big(b^{s_1}_{p_1,r_1}\to b^{s_2}_{p_2,r_2} \big)} 
&=\sup_{\|u\|_{b^{s_1}_{p_1,r_1}}\leq 1}
\|(\varphi\circ Z) \, u\|_{b^{s_2}_{p_2,r_2}}\\
&=\sup_{\|u\|_{b^{s_1}_{p_1,r_1}}\leq 1}
\|(\varphi\,(u\circ Z^{-1}))\circ Z\|_{b^{s_2}_{p_2,r_2}}\\
&\leq C_{Z,2} \sup_{\|u\|_{b^{s_1}_{p_1,r_1}}\leq 1}
\|\varphi\,(u\circ Z^{-1})\|_{b^{s_2}_{p_2,r_2}}\\
&\leq C_{Z,2} \|\varphi\|_{\cM(b^{s_1}_{p_1,r_1}\to b^{s_2}_{p_2,r_2})}
\sup_{\|u\|_{b^{s_1}_{p_1,r_1}}\leq 1}\|u\circ Z^{-1}\|_{b^{s_1}_{p_1,r_2}}\\
&\leq C_{Z^{-1},1} C_{Z,2} 
 \|\varphi\|_{\cM(b^{s_1}_{p_1,r_1}\to b^{s_2}_{p_2,r_2})}.\end{aligned}
$$

To prove the last item, it suffices to check that  if $\varphi$ belongs to $\cE'\cap \cM\big(B^{s_1}_{p_1,r_1} \rightarrow B^{s_2}_{p_2,r_2} \big),$ then $\varphi$ is also in the multiplier space between the general type Besov spaces. Take  $u \in b^{s_1}_{p_1,r_1}$ with compact support, and  some smooth and compactly supported nonnegative
cut-off function  $\psi$ satisfying $\psi \equiv 1$ on ${\rm Supp}\,\varphi.$  Then from Proposition \ref{prop:equiv_besov} and \eqref{eq:def_norm_multiplier}, we have
\begin{align*}
\|\varphi u\|_{b^{s_2}_{p_2,r_2}} &=\|\varphi \psi u\|_{b^{s_2}_{p_2,r_2}} \lesssim \|\varphi \psi u\|_{B^{s_2}_{p_2,r_2}} \\
&\lesssim \|\varphi\|_{\cM\big(B^{s_1}_{p_1,r_1} \rightarrow B^{s_2}_{p_2,r_2} \big)} \| \psi u\|_{B^{s_1}_{p_1,r_1}}\\
&\lesssim \|\varphi\|_{\cM\big(B^{s_1}_{p_1,r_1} \rightarrow B^{s_2}_{p_2,r_2} \big)} \|\psi u\|_{b^{s_1}_{p_1,r_1}}\\
&\lesssim \|\varphi\|_{\cM\big(B^{s_1}_{p_1,r_1} \rightarrow B^{s_2}_{p_2,r_2} \big)} \|\psi\|_{\cM\big(b^{s_1}_{p_1,r_1} \big)}\| u\|_{b^{s_1}_{p_1,r_1}}.
\end{align*}
For the last inequality, we used  $\cC_c^\infty \hookrightarrow \cM\big(b^{s_1}_{p_1,r_1}\big)$ (see \cite[Corollary 2.1.1]{DanM2015}). 
%If we exchange the roles of $B^{s_k}_{p_k, r_k}$  and $b^{s_k}_{p_k, r_k}$ with $k=1,2,$ then the inverse embedding holds.
\end{proof}

\section{Commutator Estimates}

We here recall and  prove some  commutator estimates that were crucial in this paper. All of them strongly rely on continuity results in Besov spaces for the paraproduct and remainder operators, and on  the following classical result (see e.g.  \cite[Section 2.10]{BCD2011}).
\begin{lemma}\label{lemma:ce1} 
Let $A:\R^N\setminus\{0\}\to\R$ be a smooth function, homogeneous of degree $m.$ Let $(\ep,s,p,r,r_1,r_2,p_1,p_2) \in ]0,1[ \times \R \times [1, \infty]^6$ with $\frac{1}{p}=\frac{1}{p_1}+\frac{1}{p_2},$
$\frac1r=\frac1{r_1}+\frac1{r_2}$  and 
$$s-m+\ep <\frac{N}{p} \quad \mbox{or} \quad \Big\{s-m+\ep <\frac{N}{p} ~~\mbox{and}~~ r=1\Big\}\cdotp$$

 There exists a constant $C$ depending only on $s, \ep, N$ and $A$ such that,
$$\|[\hT_g, A(D)]u\|_{\hB^{s-m+\ep}_{p,r}} \leq C \|\nabla g\|_{\hB^{\ep-1}_{p_1,r_1}}\|u\|_{\hB^{s}_{p_2,r_2}}.$$
\end{lemma}

%\begin{rmk} A similar inequality  holds for time-dependent distributions, as may be seen by   following 
%the proof of  Lemma \ref{lemma:ce1}, treating the time as a parameter, and applying (time) H\"older inequality 
%when appropriate. 
%For example, for any  $(\rho,\rho_1,\rho_2) \in [1,\infty]^3$ with $\frac{1}{\rho}=\frac{1}{\rho_1}+\frac{1}{\rho_2},$
%and $(\ep,s,p,r,p_1,p_2)$ as above, we  have
%\begin{equation}\label{es:ce2} 
%\|[\hT_g, A(D)]u\|_{\widetilde{L}^{\rho}_T (\hB^{s-m+\ep}_{p,r})} \leq C \|\nabla g\|_{\widetilde{L}^{\rho_1}_{T} (\hB^{\ep-1}_{p_1,\infty})}\|u\|_{\widetilde{L}^{\rho_2}_T (\hB^{s}_{p_2,r})}.
%\end{equation} 
%\end{rmk}

If the integer $N_0$ in the definition of Bony's paraproduct and remainder  is large enough (for instance $N_0 =4$ does), then the following fundamental lemma holds.
\begin{lemma}[Chemin-Leibniz Formula]\label{lemma:chemin_leibniz}
Let $(\ep,s,s_k,p, p_k, r,r_k) \in ]0,1[ \times \R^2 \times [1, \infty]^4$ for $k=1,2$ satisfying
\begin{equation*}
\frac{1}{p} = \frac{1}{p_1} + \frac{1}{p_2}\ \hbox{ and }\ \frac{1}{r} = \frac{1}{r_1} + \frac{1}{r_2}\cdotp
\end{equation*} 

\begin{enumerate}
\item If $s_2  < 0$ and $s_1+s_2+\ep-1 < \frac{N}{p}$ or $\{s_1+s_2+\ep-1=\frac{N}{p} ~~\mbox{and}~~r=1\},$ then we have 
\begin{equation*}
\|\hcT_X \hT_g f - \hT_g \hcT_X f - \hT_{ \hcT_X g}f \|_{\hB^{s_1+s_2+\ep-1}_{p,r}} \leq C \|X\|_{\hsC^{\ep}} \|f\|_{\hB^{s_1}_{p,r_1}}\|g\|_{\hB^{s_2}_{\infty,r_2}}.
\end{equation*}
The above inequality still  holds in the limit case $s_2=0,$ if one  replaces $\|g\|_{\hB^{0}_{\infty,r_2}}$ by $\|g\|_{\hB^{0}_{\infty,r_2}\cap L^\infty}$.
\item If $s_1+s_2+\ep-1 \in ]0, \frac{N}{p}[$ or $\{s_1+s_2+\ep-1=\frac{N}{p} ~~\mbox{and}~~r=1\},$ then  we have 
\begin{equation*}
\|\hcT_X \hR(f,g)-\hR(\hcT_X f, g)-\hR(f, \hcT_X g)\|_{\hB^{s_1+s_2+\ep-1}_{p,r}} \leq C \|X\|_{\hsC^\ep} \|f\|_{\hB^{s_1}_{p_1,r_1}}\|g\|_{\hB^{s_2}_{p_2,r_2}}.
\end{equation*} 
The above inequality still holds in the limit case $s_1+s_2+\ep -1 =0,$ $r=\infty$ and  $\frac{1}{r_1} + \frac{1}{r_2}=1.$
\end{enumerate}
\end{lemma}

\begin{proof} This is a mere adaptation of \cite{DanZhx2016} to the homogeneous framework.
The proof is based on a  generalized Leibniz formula for para-vector field operators which was derived by J.-Y. Chemin in \cite{Che1988}.  More precisely, define the following Fourier multipliers
\begin{equation*}
\hDelta_{k,j} := \varphi_{k}(2^{-j}D)  \quad\ \,\hbox{ with }\, \varphi_{k}(\xi):= i \xi_k \varphi(\xi) \,\, 
 \hbox{ for }\  k \in\{1,\cdots,N\}\ \hbox{ and } \,\, j \in  \Z. 
\end{equation*}
Then we have
\begin{align*}
 \hcT_X \hT_g f & = \sum_{j\in\Z} (\hS_{j-N_0} g \hcT_{X}\hDelta_{j} f + \hDelta_j f \hcT_X \hS_{j-N_0} g) + \sum_{j\in\Z}(\hT_{1,j} + \hT_{2,j})\\
 &= \hT_g \hcT_X f +\hT_{ \hcT_X g}f + \sum_{\genfrac{}{}{0pt}{}{j\in\Z}{\alpha=1,\ldots,4}} \hT_{\alpha, j},
\end{align*}
where 
\begin{align*}
\hT_{1,j} &:=  \sum_{\genfrac{}{}{0pt}{}{j \leq j' \leq j+1}{j-N_0-1 \leq j''\leq j'-N_0-1}} 2^{j'}  \hDelta_{j''}X^k \big( \hDelta_{k,j'} (\hDelta_j f \hS_{j-N_0} g)-\hDelta_{k,j'}\hDelta_j f \hS_{j-N_0}g \big),\\
\hT_{2,j} &:= \sum_{\genfrac{}{}{0pt}{}{j'\leq j- 2}{j'-N_0 \leq j''\leq j-N_0 -2}} 2^{j'}  \hDelta_{j''}X^k   (\hDelta_j f)  \hDelta_{k,j'} \hS_{j-N_0}  g,\\
\hT_{3,j} & := \hS_{j-N_0} g [\hT_{X^k}, \hDelta_j] \d_k f,\\
\hT_{4,j} & := \hDelta_j f [\hT_{X^k}, \hS_{j-N_0}] \d_k g.
\end{align*}
Bounding $\hT_{1,j}$ and $\hT_{2,j}$  stems from  the definition 
of Besov norms, and Lemmas 2.99,  2.100  of  \cite{BCD2011} allow to bound $\hT_{3,j}$ and $\hT_{4,j}$
provided  $\ep<1$. 
\medbreak
In order to prove the second item, let us set $$A_{j,j'} := \big\{j-N_0-1,\cdots,
 j'-N_0-1\big\} \cup \big\{j'-N_0,\cdots, j-N_0-2\big\}\cdotp$$
We have
\begin{align*}
 \hcT_X \hR(f, g) & = \sum_{j\in\Z} (\thDelta_{j}g \hcT_{X}\hDelta_{j} f + \hDelta_j f \hcT_X \thDelta_{j} g) + \sum_{j\in\Z}(\hR_{1,j} + \hR_{2,j})\\
 &= \hR(\hcT_X f, g)+\hR(f, \hcT_X g) + \sum_{\genfrac{}{}{0pt}{}{j\in\Z}{\alpha=1,\ldots,4}} \hR_{\alpha, j},
\end{align*}
where, denoting $\thDelta_j:=\hDelta_{j-N_0}+\cdots+\hDelta_{j+N_0},$ 
\begin{align*}
\hR_{1,j} &:=  \sum_{\genfrac{}{}{0pt}{}{|j'-j| \leq N_0 + 1}{j''\in A_{j,j'}}} {\rm sgn}(j'-j+1) 2^{j'}  \hDelta_{j''}X^k \big( \hDelta_{k,j'} (\hDelta_j f \thDelta_{j} g) - \hDelta_j f \hDelta_{k,j'} \thDelta_{j}g \big)\\
& \hspace{7cm}+ \sum_{\genfrac{}{}{0pt}{}{j-1 \leq j' \leq j}{j'-N_0 \leq j'' \leq j-N_0}} 2^{j'} \hDelta_{j''}X^k(\hDelta_{k,j'} \hDelta_j f) \thDelta_{j} g,\\
\hR_{2,j} &:= \sum_{\genfrac{}{}{0pt}{}{j' \leq j-N_0-2}{j'-N_0 \leq j''\leq j-N_0-2}} 2^{j'}  \hDelta_{j''}X^k   \hDelta_{k,j'} (\hDelta_j f \thDelta_{j}  g)   ,\\
\hR_{3,j} & := \thDelta_j g [\hT_{X^k}, \hDelta_j] \d_k f,\\
\hR_{4,j} & := \hDelta_j f [\hT_{X^k}, \thDelta_j] \d_k g.
\end{align*}
Here again, bounding $\hR_{1,j}$ and $\hR_{2,j}$ follows from the definition of
Besov norms, while Lemma 2.100 of  \cite{BCD2011} allows to bound $\hR_{3,j}$ and $\hR_{4,j}.$
\end{proof}
 
\begin{prop}\label{prop:ce_dINS}
Let $(\ep,p)$ be in $]0,1[\times [1,\infty].$ 
Consider a couple of vector fields $(X,v)$  in 
the space  $$\bigl(L^{\infty}_{loc}(\R_+; \hsC^{\ep}) \bigr)^N \times \big(L^{\infty}_{loc}(\R_+ ; \hB^{\frac{N}{p}-1}_{p,1})\cap L^{1}_{loc}(\R_+ ;\hB^{\frac{N}{p}+1}_{p,1})\big)^N,$$
 satisfying   $\div v=0$  and the transport equation 
 \begin{equation}\label{eq:prop_ce_dINS}
\left\{
\begin{array}{l}
(\d_t+v\cdot \nabla)X = \d_X v,\\[1ex]
X|_{t=0} = X_0.
\end{array}
\right.
\end{equation}
 If in addition 
 \begin{equation}\label{cdt:p_Xbis}
 \frac Np > 2-\ep,  \quad\hbox{or }\   \frac Np >1-\ep \: \hbox{ and }\: \div X\equiv0.
 \end{equation}  
then  there exists a constant $C$ such that:
\begin{multline}\label{es:TX}
 \|[\hcT_X, \d_t+v\cdot \nabla ]v\|_{\hB^{\frac{N}{p}+\ep-2}_{p,1}} \leq C(\|X\|_{\hsC^\ep} \|v\|_{\hB^{\frac{N}{p}+1}_{p,1}} \|v\|_{\hB^{\frac{N}{p}-1}_{p,1}} \\+\|v\|_{\hsC^{-1}}\|\hcT_{X}v\|_{\hB^{\frac{N}{p}+\ep}_{p,1}}+\|v\|_{\hB^{\frac{N}{p}+1}_{p,1}}\|\hcT_{X}v\|_{\hsC^{\ep-2}}).
\end{multline}
 \end{prop}

\begin{proof} 
This is essentially the proof  of  \cite[Proposition A.5]{DanZhx2016}. For the reader convenience, we here give a sketch of it.
Because   $\div v=0,$ we may write 
\begin{align*}
[\hcT_X, \d_t+v^{\ell} \d_\ell ]v & = -v^\ell\d_\ell \hT_{X^k}\d_kv- \hT_{\d_t X^k}\d_kv 
+ \hT_{X^k}\d_k(v^{\ell}\d_\ell v)\\  
&=- \hT_{\d_t X^k}\d_kv + \d_\ell\hcT_{X}(v^{\ell} v)-\hcT_{\d_\ell X}(v^{\ell} v) - v^{\ell} \d_{\ell} \hcT_Xv.
\end{align*}
Hence, decomposing  $v^{\ell} v$ according to Bony's decomposition, we discover that  
$$
[\hcT_X, \d_t+v^{\ell} \d_\ell ]v = \sum_{\alpha=1}^{\alpha=5} \hR_{\alpha}
$$
with
\begin{align*}
\hR_1 &:=- \hT_{\d_t X^k}\d_kv, &\hR_2 &:= \d_{\ell}(\hcT_X \hT_{v^{\ell}}v +\hcT_X \hT_{v}v^{\ell} ),\\
\hR_3 &:= \d_{\ell} \hcT_{X} \hR(v^{\ell},v),   &\hR_4& :=-\hcT_{\d_\ell X}(v^{\ell} v),\\
\hR_5 &:=- v^{\ell} \d_{\ell} \hcT_Xv.
\end{align*}
Now, it suffices to check that all the terms $\hR_\alpha$ may be bounded by the r.h.s. of \eqref{es:TX}.

\subsubsection*{$\bullet$ \underline{Bound of $\hR_1$}:}

From the equation \eqref{eq:prop_ce_dINS}, we have 
\begin{equation*}
\hR_1 = \hT_{v\cdot \nabla X^k} \d_k v - \hT_{\d_X v^k} \d_k v.
\end{equation*}
Hence using standard continuity results for the paraproduct, we deduce that
$$
\|\hR_1\|_{\hB^{\frac{N}{p}+\ep-2}_{p,1}} \lesssim \|\nabla v\|_{\hB^{\frac{N}{p}}_{p,1}}\bigl(\|v \cdot \nabla X\|_{\hsC^{\ep-2}} + \|\d_X v\|_{\hsC^{\ep-2}}\bigr).
$$
Keeping in mind \eqref{cdt:p_Xbis}, the last term may be bounded according to \eqref{es:dXf_TXf},
after using the embedding  $\hB^{\frac{N}{p}+\ep-2}_{p,1}(\R^N) \hookrightarrow \hsC^{\ep-2}(\R^N).$
We get
\begin{equation*}
\|\d_X v- \hcT_X v\|_{\hsC^{\ep-2}} \lesssim \|\nabla v\|_{\hB^{\frac{N}{p}-2}_{p,1}} \|X\|_{\hsC^\ep}.
\end{equation*}
  
As for the first term, we use the fact $\div v=0$ and the following decomposition
$$
v\cdot\nabla X=\hcT_{v}  X + \hT_{\d_{\ell} X } v^\ell +\d_\ell \hR (v^{\ell},X),
$$
which allow to get, as long as \eqref{cdt:p_Xbis} holds
\begin{equation*}
\|\hR_1\|_{\hB^{\frac{N}{p}+\ep-2}_{p,1}}  \lesssim \|\nabla v\|_{\hB^{\frac{N}{p}}_{p,1}}(\|v\|_{\hB^{\frac{N}{p}-1}_{p,1}} \|X\|_{\hsC^{\ep}} + \|\hcT_X v\|_{\hsC^{\ep-2}}).
\end{equation*}
\subsubsection*{$\bullet$  \underline{Bound of $\hR_2$}:}

Due to Lemma \ref{lemma:chemin_leibniz} (i) and continuity of paraproduct operator, we have
\begin{equation*}
\|\hR_2\|_{\hB^{\frac{N}{p}+\ep-2}_{p,1}} \lesssim \|X\|_{\hsC^{\ep}} \|v\|_{\hB^{\frac{N}{p}+1}_{p,1}} \|v\|_{\hsC^{-1}} +\|v\|_{\hsC^{-1}} \|\hcT_{X}v\|_{\hB^{\frac{N}{p}+{\ep}}_{p,1}}+\|v\|_{\hB^{\frac{N}{p}+1}_{p,1}}\|\hcT_{X}v\|_{\hsC^{\ep-2}}.
\end{equation*}
\subsubsection*{$\bullet$ \underline{Bound of $\hR_3$}:}

Applying Lemma \ref{lemma:chemin_leibniz} (ii) and continuity of remainder operator under the condition
 $\frac{N}{p} +\ep-1>0$ yields
\begin{equation*}
\|\hR_3\|_{\hB^{\frac{N}{p}+\ep-2}_{p,1}} \lesssim \|X\|_{\hsC^{\ep}} \|v\|_{\hB^{\frac{N}{p}+1}_{p,1}}\|v\|_{\hsC^{-1}} + \|v\|_{\hB^{\frac{N}{p}+1}_{p,1}} \|\hcT_X v\|_{\hsC^{\ep-2}}.
\end{equation*}

\subsubsection*{$\bullet$ \underline{Bound of $\hR_4$}:}
{}From  Bony decomposition \eqref{eq:bony}, it is easy to get
\begin{equation*}
\|v^l v\|_{\hB^{\frac{N}{p}}_{p,1}} \lesssim \|v\|_{\hsC^{-1}}\|v\|_{\hB^{\frac{N}{p}+1}_{p,1}}.
\end{equation*}
Hence
\begin{equation*}
\|\hR_4\|_{\hB^{\frac{N}{p}+\ep-2}_{p,1}} \lesssim \|\nabla X\|_{\hsC^{\ep-1}}\|v\|_{\hsC^{-1}}\|v\|_{\hB^{\frac{N}{p}+1}_{p,1}}. 
\end{equation*}

\subsubsection*{$\bullet$ \underline{Bound of $\hR_5$}:}
Applying Bony decomposition and using  that $\div v=0$ and $\frac{N}{p} + \ep > 1$ give
\begin{equation*}
\|\hR_5\|_{\hB^{\frac{N}{p}+\ep-2}_{p,1}} \lesssim \|v\|_{\hsC^{-1}}\|\hcT_{X}v\|_{\hB^{\frac{N}{p}+\ep}_{p,1}}+\|v\|_{\hB^{\frac{N}{p}+1}_{p,1}}\|\hcT_{X}v\|_{\hsC^{\ep-2}}.
\end{equation*}
Combining the above  estimates for all $\hR_{\alpha}$, with $\alpha=1,\dots,5$ yields \eqref{es:TX}. 
\end{proof}

Another consequence of Lemma \ref{lemma:chemin_leibniz} is the following estimate of $\div (X f g)$:
\begin{prop} \label{prop:dXfg_ND}
Let  $(s,p,r)$ be  in $]0,1[\times [1,\infty]^2,$ and $\eta,$ in $]0,1-s[.$ 
Consider a  bounded vector field   $X$ and two bounded functions $f,g$ satisfying
\begin{equation*}
X \in  \big(\hB^s_{p,r} (\R^N) \cap \sC^{s+\eta} (\R^N)\big)^{N},\,\,\, (f, g )\in \hB^s_{p,r} (\R^N) \times \hB^{-\eta}_{p, r} (\R^N) \,\,\,\hbox{and}\,\,\,  \d_X g \in \hB^{s-1}_{p,r} (\R^N).
\end{equation*}
If in addition  $\div X $ belongs to $\cM\big(\hB^{s}_{p,r}(\R^N) \rightarrow \hB^{s-1}_{p,r}(\R^N)\big),$ 
and  there exists some $q\in[1,p[$ 
such that 
\begin{equation}\label{cdt:s_pq}
 \div X \in \hB^{s_{p,q}}_{q,r}(\R^N)\ \hbox{  with }\ s_{p,q} := s-1 + N \big(\frac{1}{q} - \frac{1}{p}\big)>0,   
\end{equation}
 then  we have $\div (X f g) \in \hB^{s-1}_{p,r}(\R^N),$ and  the following estimate holds true:
 \begin{multline*}
 \|\div (X f g)\|_{\hB^{s-1}_{p,r}}\lesssim \|X\|_{\hB^s_{p,r} \cap \sC^{s+\eta}} \|f\|_{ L^\infty \cap \hB^s_{p,r}}\|g\|_{L^\infty \cap \hB^{-\eta}_{p,r}}  
+  \|f\|_{L^\infty} \|\d_X g\|_{\hB^{s-1}_{p,r}} \\
 + \|\div X\|_{\hB^{s_{p,q}}_{q,r} \cap \cM(\hB^{s}_{p,r}\rightarrow \hB^{s-1}_{p,r})}  \|g\|_{L^\infty}  \|f\|_{\hB^{s}_{p,r} \cap L^\infty}.
 \end{multline*}
\end{prop}
\begin{proof}
In light of Bony's decomposition \eqref{eq:bony}, and  denoting 
$\hT'_g f := \hT_g f + \hR(f,g),$ we can decompose  $\div ( X fg)$ into
\begin{equation*}
\div (Xfg) = \div \big( \hT'_{fg} X + \hT_X (fg) \big) = \sum_{\alpha=1}^{4} \hF_{\alpha},
\end{equation*}
where 
\begin{align*}
\hF_1 &:= \div (\hT'_{fg} X),  &\hF_2 &:= \hT_{\div X} (fg),\\
\hF_3 &:= \hcT_X \hT'_g f,  &\hF_4 &:= \hcT_X \hT_f g.
\end{align*}

\subsubsection*{$\bullet$ \underline{Bound of $\hF_1$}:} As $s>0,$  standard
continuity results for $\hR$ and $\hR$ give  
\begin{equation*}
\|\hF_1\|_{B^{s-1}_{p,r}} \lesssim \|\hT'_{fg} X\|_{\hB^{s}_{p,r}} \lesssim \|f\|_{L^\infty}\|g\|_{L^\infty} \|X\|_{\hB^{s}_{p,r}}.  
\end{equation*}

\subsubsection*{$\bullet$ \underline{Bound of $\hF_2$}:}
Thanks to continuity results for $\hT',$   we have for $s <1,$   
\begin{equation*}
\|\hF_2\|_{\hB^{s-1}_{p,r}} \lesssim \| \div X\|_{\hB^{s-1}_{p,r}} \|f\|_{L^\infty}\|g\|_{L^\infty}. 
\end{equation*}
\subsubsection*{$\bullet$ \underline{Bound of $\hF_3$}:}
Because $X$ and $g$ are in $L^\infty$ and $s>0,$  
we readily have
\begin{equation*}
\|\hF_3\|_{\hB^{s-1}_{p,r}} \lesssim \|X\|_{L^\infty} \|\hT'_g f\|_{\hB^{s}_{p,r}} \lesssim \|X\|_{L^\infty}\|g\|_{L^\infty} \|f\|_{\hB^{s}_{p,r}}.
\end{equation*}
\subsubsection*{$\bullet$ \underline{Bound of $\hF_4$}:}
Because $0<s<s+\eta<1,$ 
Lemma \ref{lemma:chemin_leibniz} and continuity results for the paraproduct imply that
$$\begin{aligned}
\|\hcT_X \hT_f g\|_{\hB^{s-1}_{p,r}} &\lesssim  \|X\|_{\hsC^{s+\eta}} \|f\|_{L^\infty}\|g\|_{\dot B^{-\eta}_{p,r}} + \| \hT_f \hcT_X g \|_{\hB^{s-1}_{p,r}} +\| \hT_{ \hcT_X f}g \|_{\hB^{s-1}_{p,r}}\\
&\lesssim \|X\|_{\hsC^{s+\eta}}  \|f\|_{L^\infty}\|g\|_{\dot B^{-\eta}_{p,r}} +  \|f\|_{L^\infty} \|\hcT_X g\|_{\hB^{s-1}_{p,r}}  + \|g\|_{L^\infty}\|\hcT_X f \|_{\hB^{s-1}_{p,r}}.
\end{aligned}$$
To bound the last term, one may use the decomposition
$$\hcT_X f=\div(\hT_Xf)-f\div X+\hT_f\,\div X+\hR(f,\div X).$$
Hence using continuity results for $\hR$ and $\hT$ and the fact that 
 $(s_{p,q},q)$  satisfies  \eqref{cdt:s_pq},
$$
\|\hcT_X f\|_{\dot B^{s-1}_{p,r}}\lesssim
\|f\|_{\dot B^s_{p,r}}\bigl(\|X\|_{L^\infty}+ \|\div X\|_{\cM(\hB^{s}_{p,r}\rightarrow \hB^{s-1}_{p,r})}\bigr)
+\|f\|_{L^\infty}\|\div X\|_{ \hB^{s_{p,q}}_{q,r}}.
$$
Finally,  to bound the term with $\hcT_Xg,$ we use the fact that 
 $$\d_X g - \hcT_X g = \hT_{\nabla g}\cdot  X + \div \hR(X, g) - \hR(\div X, g),$$ 
 whence
\begin{equation}\label{es:dXh}
\|\d_X g - \hcT_X g\|_{\hB^{s-1}_{p,r}} \lesssim \|g\|_{L^\infty} \big(\|X\|_{\hB^s_{p,r}} + \|\div X\|_{\hB^{s_{p,q}}_{q,r}} \big).
\end{equation}
This completes the proof of the proposition.
\end{proof}
Proposition \ref{prop:dXfg_ND} above reveals that the bounded function $g$ may behave like some element in $\cM(\hB^{s}_{p,r})$ under a suitable additional structure assumption. 
If in addition $g$ has compact support, then one can relax a bit the regularity of  $X$ and $f$ to study $\d_X (fg),$ 
and get the following   generalization of \cite[Lemma A.6]{Dan1999}.
\begin{cor} \label{cor:dXfg_3D}
Consider a divergence-free vector field $X$ with coefficients in $B^\alpha_{\infty,1},$
and some function $f$ in $B^{\alpha'}_{\infty,1}$ with $0<\alpha,\alpha'<1.$ 
 Let $g \in L^\infty$ be compactly supported and satisfy  $\d_X g  \in \dot B^{\min\{\alpha,\alpha'\}-1}_{p,1}$ for some $p \in [1,\infty].$ Then we have $\d_X (fg) \in \dot B^{\min\{\alpha,\alpha'\}-1}_{p,1}.$     
\end{cor}
\begin{proof}
Let  $\psi \in \cC^{\infty}_c$ be a cut-off function such that 
$\psi\equiv 1$ near ${\rm Supp}\, g.$ Denote $(\tX, \tf):= (\psi X, \psi f).$ 
From Proposition \ref{prop:besov-holder}, we know that $\tf$
 and $\tX$ are  in $\dot B^{\min(\alpha,\alpha')}_{q,1} \cap L^\infty$
for any $q \in [1,\infty].$ 
It is also  clear that  $\d_X (fg) = \d_{\tX} (\tf g)\in \dot B^{\min\{\alpha,\alpha'\}-1}_{p,1}.$ 
Hence applying Proposition  \ref{prop:dXfg_ND}
gives the result. 
\end{proof}
\end{appendix}

\section*{Acknowledgement}  
Part of this work has been performed  during a visit of  the second author in  Peking University. Zhifei Zhang and Chao Wang are warmly thanked for their welcome and stimulating discussions during this period. 
The first author is partially supported by  ANR-15-CE40-0011.
The second author is  granted by the  \emph{R\'eseau de Recherche Doctoral de  Math\'ematiques de l'\^Ile de France} (RDM-IdF).

\end{document}